\documentclass[12pt]{amsart}
\usepackage{verbatim}
\usepackage{natbib}
\usepackage{amsmath}
\usepackage{amssymb}
\usepackage{calc}
\usepackage{amsmath}
\usepackage{amssymb}
\usepackage{amsthm}
\usepackage{relsize}
\usepackage{bm}
\usepackage{hyperref}
\usepackage{tikz-cd}
\usepackage[utf8]{inputenc}
\usepackage[english]{babel}
\usepackage[margin=1in]{geometry}
\usepackage{graphicx}
\usepackage{subfigure}

\newtheorem{theorem}{Theorem}[section]

\newtheorem{proposition}[theorem]{Proposition}
\newtheorem{lemma}[theorem]{Lemma}
\newtheorem{corollary}[theorem]{Corollary}

\newtheorem{question}[theorem]{Question}

\theoremstyle{definition}
\newtheorem{definition}[theorem]{Definition}
\newtheorem{condition}[theorem]{Condition}

\newtheorem{remark}[theorem]{Remark}

\title{Objects of a Phantom on a Rational Surface}
\author{Amal Mattoo}

\begin{document}
	\maketitle 
	
	\begin{abstract}
		Johannes Krah showed that the blowup of $\mathbf{P}^{2}$ in $10$ general points admits a phantom subcategory. We construct three types of objects in such a phantom: a strong generator, projections of skyscraper sheaves, and a family of objects with two nonzero cohomology sheaves. We study the deformation theory of these objects to show that the phantom contains rich geometry, such as encoding the blowdown map to $\mathbf{P}^{2}$. We also show that there exists a co-connective dg-algebra whose derived category is a phantom. 
	\end{abstract}
	
	\tableofcontents
	
	\section{Introduction}
		A \emph{phantom} on a smooth projective variety $X$ is a nonzero admissible subcategory $\mathcal{A}\hookrightarrow D^{b}_{\text{Coh}}(X)$ such that $K_{0}(\mathcal{A})=0$. This condition implies that all additive invariants of $\mathcal{A}$ vanish. 
		
		Gorchinskiy-Orlov \cite{Gorchinskiy_2013} constructed the first example of a phantom on the product of two general type surfaces, and further examples were found on determinantal Barlow surfaces in \cite{B_hning_2015}. Recently Krah \cite{Krah2023} discovered a phantom on a much simpler variety, a rational surface obtained by blowing up $\mathbf{P}^{2}$ in $10$ general points, exhibited as the orthogonal complement of a non-full maximal length exceptional collection.  
		
		As \cite{B_hning_2015} noted, it would be interesting to ``lay hands'' on a phantom by constructing explicit objects. Many techniques used in the study of derived categories, such as Hodge theory and stability conditions, are inapplicable without nontrivial additive invariants, so studying explicit objects appears to be one of the only available tools. 
		
		In this paper we will construct explicit objects in a slight modification of Krah's phantom on a rational surface. We will study their homological properties and deformations, with a particular view to seeing how much geometry of $X$ can be recovered from $\mathcal{A}$.
		
		Our main tool will be the projection functor $i^{*}$, which is the left adjoint to the inclusion $i:\mathcal{A}\hookrightarrow D^{b}_{\text{Coh}}(X)$. A key technical step is to find a manageable way to compute with this projection. In particular, Proposition \ref{projection-complex} gives an explicit complex for the projection of an object, and Corollary \ref{spectral-sequence} gives a spectral sequence computing the Homs between projections. 
		
		These constructions have not appeared in the literature, but they are quite similar to the normal Hochschild complex and normal Hochschild spectral sequence of Kuznetsov \cite{Kuznetsov2012}. He used a dg-enhancement and the bar resolution, but we will take a more down-to-earth approach using injective resolutions. Though messier, our method is perhaps more transparent as we will inductively apply exact triangles for projecting past exceptional objects.  
		
		In Section \ref{section-generator}, we will construct a strong generator $\mathcal{Q}$ for the phantom. While this construction will not be too explicit, we will show that $\text{Hom}^{*}(\mathcal{Q},\mathcal{Q})$ is concentrated in non-negative degree, which will help answer affirmatively a question of Ben Antieau:
		\begin{theorem}[Theorem \ref{Antieau}]
			There exists a co-connective dg-algebra whose derived category is a phantom.
		\end{theorem}
		
		Then in Section \ref{section-skyscraper}, we will study projections of skyscraper sheaves $i^{*}\kappa(x)$. We will compute the $\text{Hom}$ spaces between them (Lemma \ref{Hom(k(x),k(x))} and Corollary \ref{different-points}), which are the same for all $x\in X$, as well as composition maps of $\text{Hom}^{1}(i^{*}\kappa(x),i^{*}\kappa(x))$ (Lemma \ref{general-compositions}), which depend on $x$. It turns out $\text{Hom}^{*}(i^{*}\kappa(x),i^{*}\kappa(x))$ vanishes for $*<0$ and is $\mathbf{C}$ for $*=0$. This means the map $p:x\mapsto i^{*}\kappa(x)$ from $X$ to the moduli stack of objects of $\mathcal{A}$ lands in the locus of simple universally gluable objects $s\mathcal{M}(\mathcal{A})$, which is an algebraic stack and a $\mathbf{G}_{m}$-gerbe. 
		
		We will use the above computations to study the deformation theory of $i^{*}\kappa(x)$ and understand the image of $p$ in the reduced coarse moduli space $sM(\mathcal{A})$ of $s\mathcal{M}(\mathcal{A})$. Remarkably, for generic $x\in X$, we will see that the only unobstructed deformations are those arising geometrically from deforming the point $x$ (Lemma \ref{obstructed}). This will imply:
		\begin{proposition}[\ref{birational}]
			The map $p:X\to sM(\mathcal{A})$ is a birational map from $X$ to an irreducible component of $sM(\mathcal{A})$.
		\end{proposition}  
	
		Studying the \emph{special locus} of $x\in X$ that could have additional deformations, we see that it reflects the geometry of the blowdown map $\pi\colon X\to\mathbf{P}^{2}$ (Lemma \ref{linear-systems}), which is not intrinsic to $X$ but is used in the construction of $\mathcal{A}$. Thus, we can show:
		\begin{theorem}(Theorem \ref{blowdown})
			The data of $i\colon \mathcal{A}\hookrightarrow D^{b}_{\textnormal{Coh}}(X)$ uniquely characterizes the blowdown map $\pi\colon X\to\mathbf{P}^{2}$.
		\end{theorem}
	
		In Section \ref{section-two-cohomology-sheaves}, we will construct a family of relatively simple objects $\mathcal{P}$ that we can describe very explicitly. Such $\mathcal{P}$ have exactly two non-zero cohomology sheaves (Lemma \ref{two-cohomology-sheaves}), which is the minimal possible number, as a phantom category cannot contain a sheaf. And computing $\text{Hom}^{*}(\mathcal{P},\mathcal{P})$ (Proposition \ref{Hom(P,P)}) shows that they are also simple and universally gluable. 
		
		We will construct $\mathcal{P}$ as $i^{*}(\iota_{*}\mathcal{L})$ for a specially chosen curve $\iota:C\to X$ and line bundle $\mathcal{L}$ on $C$. It will take some work to show that the desired curve $C$ exists (Corollary \ref{existence}), and the proof will engage deeply with the geometry of $X$. 
		
		Furthermore, we will show that deformations of $\mathcal{P}$ correspond precisely to deformations of $C$ and $\mathcal{L}$:
		\begin{proposition}[Proposition \ref{Hom(P,P)}, Proposition \ref{normal-bundle}]
			There is an exact sequence 
			$$0\to H^{1}(C,\mathcal{O}_{C})\to\textnormal{Hom}_{X}^{1}(\mathcal{P},\mathcal{P})\to H^{0}(C,\mathcal{N}_{X/C})\to 0$$
			where $\mathcal{N}_{X/C}$ is the normal bundle.
		\end{proposition}
		Thus, we can characterize another irreducible component of $sM(\mathcal{A})$.
		
		Lemma \ref{|-3F|} shows the existence of the lowest possible degree curve $C$ from which $\mathcal{P}$ could be constructed. A possible advantage of this result is that such $\mathcal{P}$ appears to be one of the homologically simplest objects in the phantom, and showing this precisely would intrinsically characterize $\mathcal{P}$ in $\mathcal{A}$. If we had any kind of intrinsic characterization of objects of the phantom, studying their moduli as we do here would allow us to extract information about the geometry of $X$ using only the triangulated category structure of $\mathcal{A}$. This will hopefully be an avenue for future work.
		
		\vspace{12pt}
		
		\noindent \textbf{Acknowledgements.}
		Thank you to my advisor, Johan de Jong, for his invaluable guidance, as well as to Alex Perry, James Hotchkiss, and Johannes Krah, and Nicol\'as Vilches for many very helpful conversations.  
	\section{Preliminaries}

	\subsection{Krah's Phantom}
		While the following construction of a phantom on a rational surface is almost that discovered by Johannes Krah in \cite{Krah2023}, we will slightly modify his method to study a different phantom subcategory of $D^{b}_{\text{Coh}}(X)$ (see Remark \ref{Krah-relationship} for the precise relationship). While Section \ref{section-generator} and Section \ref{section-skyscraper} of this paper would be unchanged using his original construction, the results of Section \ref{section-two-cohomology-sheaves} would have to be modified and seem to be less elegant. In particular, the different embedding into $D^{b}_{\text{Coh}}(X)$ affects the cohomology sheaves of the objects we will construct.
		
		Let $X$ be the blowup of $\mathbf{P}^{2}$ over $\mathbf{C}$ in $10$ general points, and consider the divisors
		$$D_{i}:=-6H+2\sum_{j=1}^{10}E_{j}-E_{i}\text{ and } 
		F:=-19H+6\sum_{i=1}^{10}E_{i}.$$
		These are the images of $E_{i}$ and $H$ respectively under a reflection of $\text{Pic}(X)$ that preserves the intersection product and fixes $K_{X}$.
		
		By ``10 general points'' we mean that for finitely many $|D|$, the following condition holds:
		\begin{condition}\label{generality}
			For $|D|=|dH-\sum_{i=1}^{10}m_{i}E_{i}|$ with $d,m_{i}\geq0$,
			$\dim|D|$ is minimal among all choices of $10$ points in $\mathbf{P}^{2}$. 
		\end{condition} 
		Every time we appeal to generality of the $10$ points, we will remark on which choices of $|D|$ we want to satisfy this property.
		
		We will fix the following exceptional collection, whose objects are the duals of those in \cite{Krah2023}, 
		$$\langle\mathcal{E}_{1},...,\mathcal{E}_{13}\rangle:=\langle\mathcal{O}_{X}(-2F),\mathcal{O}_{X}(-F),\mathcal{O}_{X}(-D_{1}),...,\mathcal{O}_{X}(-D_{10}),\mathcal{O}_{X}\rangle, $$
		the phantom we study will be $\mathcal{A}:=\langle\mathcal{E}_{i}\rangle_{i}^{\perp}$.		
		\begin{proposition}[\cite{Krah2023}]\label{Krah}
			The collection $\langle\mathcal{E}_{i}\rangle_{i}$	
			is exceptional, and $\mathcal{A}$ is non-zero with $K_{0}(\mathcal{A})=0$.
		\end{proposition}
		\begin{proof}
			This follows from precisely the same computations as the proof of Theorem 1.1 of \cite{Krah2023}. Indeed, Krah shows that $\langle \mathcal{E}_{13}^{\vee},...,\mathcal{E}_{1}^{\vee}\rangle$ is a non-full exceptional collection by showing $\text{Hom}^{*}(\mathcal{E}_{i}^{\vee},\mathcal{E}_{j}^{\vee})=0$ and $\text{Hom}^{0}(\mathcal{E}_{j}^{\vee},\mathcal{E}_{i}^{\vee})=0$ for $i<j$, and we have
			\begin{align*}
				\text{Hom}^{*}(\mathcal{E}_{j},\mathcal{E}_{i}) &=
				\text{Hom}^{*}(\mathcal{E}_{i}^{\vee},\mathcal{E}_{j}^{\vee}) 
			\end{align*}
			for all $i,j$, proving that our collection is exceptional of maximal length but not full.
		\end{proof}
		Our results, such as Proposition \ref{Hom(k(x),k(x))} and Proposition \ref{two-cohomology-sheaves} will provide an alternative proof that this exceptional collection is not full by constructing non-zero objects in $\mathcal{A}$.
		
		\begin{remark}\label{Krah-assumptions}
			The proof of Proposition \ref{Krah} assumes that the $10$ blown-up points are general in the sense that Condition \ref{generality} holds for the following:  
			$$|-F|,\quad |-2F|,\quad |-D_{i}|,\quad |D_{i}-F|,\quad |D_{i}-F|.$$
		\end{remark}
		\begin{remark}\label{Krah-relationship}
			Our $\mathcal{A}$ is the image of Krah's phantom in \cite{Krah2023} under the anti-equivalence 
			$$\mathcal{H}\text{om}(-,\omega_{X})\colon D^{b}_{\text{Coh}}(X)\to D^{b}_{\text{Coh}}(X),\quad
			\mathcal{K}\mapsto\mathcal{K}^{\vee}\otimes\omega_{X}.$$
			Indeed, the elements of Krah's exceptional collection are $\mathcal{E}_{i}^{\vee}$, and for any $\mathcal{K}\in D^{b}_{\text{Coh}}(X)$ we have:
			\begin{align*}
				\text{Hom}^{*}(\mathcal{E}_{i},\mathcal{K}^{\vee}\otimes\omega_{X}) &= 
				\text{Hom}^{*}(\mathcal{K}^{\vee},\mathcal{E}_{i})^{\vee} \\&=
				\text{Hom}^{*}(\mathcal{E}_{i}^{\vee},\mathcal{K})^{\vee},
			\end{align*}
			so $\mathcal{K}^{\vee}\otimes\omega_{X}\in\mathcal{A}$ if and only if $\mathcal{K}$ is in Krah's phantom.
		\end{remark}
	
		The next Lemma shows that many of the ``forward'' Homs between objects of this exceptional collection vanish, which facilitates the nice properties we will discover.
		
		\begin{lemma}\label{vanishing}
			$\textnormal{Hom}^{*}(\mathcal{E}_{i},\mathcal{E}_{j})=\mathbf{C}^{\chi(\mathcal{E}_{j}\otimes\mathcal{E}_{i}^{\vee})}[-2]$ for $0\leq i<j\leq 13$ and $\textnormal{Hom}^{*}(\mathcal{O}_{X}(-D_{i}),\mathcal{O}_{X}(-D_{j}))=0$ for $1\leq i<j\leq 10$.
		\end{lemma}
		\begin{proof}
			These $\text{Hom}^{*}(\mathcal{E}_{j}^{\vee},\mathcal{E}_{i}^{\vee})=\text{Hom}^{*}(\mathcal{E}_{i},\mathcal{E}_{j})$ are computed in \cite{Krah2023} in the Proof of Theorem 1.1.
		\end{proof}
		\begin{remark}
			The proof of Lemma \ref{vanishing} assumes Condition \ref{generality} for the same divisor classes listed in Lemma \ref{Krah-assumptions}.
		\end{remark}
	
	\subsection{Linear Systems on $\text{Bl}_{10\text{pts}}\mathbf{P}^{2}$}
		Let $X$ still be the blowup of $\mathbf{P}^{2}$ at $x_{1},...,x_{10}$, and let $|D|$ satisfy the following condition:
		\begin{condition}\label{generality-base-loci}
			$|D|=|dH-\sum_{i=1}^{10}m_{i}E_{i}|$ with $d,m_{i}\geq0$ satisfies Condition \ref{generality}, and so do all $|B|=|d'H-\sum_{i=1}^{10}m_{i}'E_{i}|$ with $0\leq d'\leq d$ and $0\leq m_{i}'\leq m_{i}$.
		\end{condition}
		
		Here we collect some results from the literature and some additional lemmas which will be useful in multiple sections below. 
		
		\begin{lemma}[Theorem 7 of \cite{MarcinDumnicki2011}]\label{Cremona}
			The linear systems $|D|$ and $|d'H-\sum_{i=1}^{10}m_{i}'E_{i}|$ have the same dimension, where
			$$d'=2d-(m_{1}+m_{2}+m_{3}),\quad
			m_{i}'=\begin{cases}
				d-(m_{1}+m_{2}+m_{3})+m_{i} & i=1,2,3 \\
				m_{i} & i>3
			\end{cases} $$
		\end{lemma}
		The operation of replacing $d,m_{i}$ with $d',m_{i}'$ is called \emph{Cremona reduction}.
		
		Now, let the symmetric group $S_{10}$ act on $\{x_{i}\}_{i=1}^{10}$ by permutation. This induces an action on $\text{Pic}(X)$ via $\sigma |H|=|H|$ and $\sigma |E_{i}|=|E_{\sigma(i)}|$.

		\begin{lemma}\label{symmetry}
			$\dim\sigma|D|$ is the same for all $\sigma$.
		\end{lemma}
		\begin{proof}
			If not, then we can find $\dim\sigma|D|
			<\dim|D|$. But recall that $\{x_{i}\}_{i}$ are in general position with respect to $|D|$, meaning that $\dim|D|$ is minimal. But replacing $\{x_{i}\}_{i}$ with $\{\sigma x_{i}\}_{i}$ reduces this dimension, which is a contradiction.
		\end{proof}
		
		\begin{lemma}\label{permutation}
			Let $\sigma\in S_{10}$ be a symmetry of $\{m_{i}\}_{i=1}^{10}$. If $B$ is in the base locus of $|D|$, then some effective divisor $B'\in\sigma|B|$ is in the base locus of $|D|$ as well.
		\end{lemma}
		\begin{proof}
			By Lemma \ref{symmetry}, and noting that we assumed $\{x_{i}\}_{i}$ are in general position with respect to $|B|$, we have $\sigma|B|$ is non-empty so some such $B'\in\sigma|B|$ exists. Then, likewise noting that $\{x_{i}\}_{i}$ are also in general position with respect to $|D-B|$,  
			\begin{align*}
				\dim |D| &= \dim |D-B| \\&=
				\dim \sigma|D-B| \\&=
				\dim |D-B'|
			\end{align*}
			implying that $B'$ is in the base locus of $D$.
		\end{proof}
			
		\begin{definition}
			A linear system $|D|$ is called \emph{special} if $h^{0}(X,\mathcal{O}_{X}(D))>0$ and $h^{0}(X,\mathcal{O}_{X}(D))>\chi(X,\mathcal{O}_{X}(D))$.  			
		\end{definition}
		Equivalently, special linear systems are those with $h^{0}(X,\mathcal{O}_{X}(D))h^{1}(X,\mathcal{O}_{X}(D))\neq 0$, meaning that the dimension of $|D|$ is greater than expected from the Euler characteristic. 
		
		The following results we cite assume $\{x_{i}\}_{i=1}^{10}$ are general. In most of our cases, the generality will follow a posteriori from Condition \ref{generality} or Condition \ref{generality-base-loci} applied to specific divisor classes.  

		\begin{lemma}\label{standard-form}
			$|D|$ is non-special if $d\geq\max_{i>j>k}\{m_{i}+m_{j}+m_{k}\}$ and all $m_{i}\leq 11$.
		\end{lemma}
		\begin{proof}
			Proposition 1.4 of \cite{CM2011} gives that such a system is not $(-1)$-special, which combined with Corollary 26 of \cite{DUMNICKI2007621} implies that it is not special.
		\end{proof}
		\begin{lemma}[\cite{Petrakiev2014}]\label{non-special}
			If $\frac{d}{m}<\frac{2280}{721}$, then $|D|$ is empty.
		\end{lemma}		
		
		\begin{lemma}[Theorem 0.1 of \cite{CM2011}]\label{174/55}
			If $\frac{d}{m}\geq \frac{174}{55}$, then $|D|$ is non-special.
		\end{lemma}

	\subsection{The Projection Spectral Sequence}

		We will use injective resolutions to compute the functor $i^{*}$. Alternative approaches would be to employ the bar resolution as in \cite{Kuznetsov2012}, or to use Postnikov systems, twisted complexes, and convolutions such as in \cite{10.1093/imrn/rnaa371}. 
%		I hope to write up these results using those alternative techniques as well.
		
		We will make use of the notion of totalization of multicomplexes, following \cite{Chardin_2023}.
		
		\begin{definition}
			An \emph{$n$-multicomplex} $C^{\underline{\bullet}}$ is a $\bigoplus_{i=1}^{n}\mathbf{Z}e_{i}$-graded object, with differentials $d^{i}$ of degree $e_{i}$ for $i=1,...,n$ that square to zero and commute with each other.
			
			Letting $\sigma:d^{i}=(-1)^{\deg}d^{i}$, the \emph{totalization} $\text{Tot}(C^{\underline{\bullet}})^{\bullet}$ is the complex whose degree $m$ term is the direct sum of the graded pieces of $C^{\underline{\bullet}}$ with total degree $m$ and whose differential is $d^{m}:=\sum_{i=1}^{n}\sigma(d^{i})$.
		\end{definition}
		Our multicomplexes will either be double complexes or come from tensor products of complexes. To agree with the sign conventions above, we define the composition map on Hom complexes as follows:
		$$\text{Hom}^{\bullet}(A,B)\otimes\text{Hom}^{\bullet}(B,C)\to\text{Hom}^{\bullet}(A,C),\quad
		f\otimes g\mapsto (-1)^{\deg f+\deg g}g\circ f.$$
		
		We show a technical lemma that will allow us to commute Homs with tensor products.
		
		\begin{lemma}\label{pullout}
			For $X$ a Noetherian scheme over the field $k$, let $V^{\bullet}$ be a bounded complex of $k$-vector spaces with finite dimensional cohomologies, and let $I^{\bullet}$ and $G^{\bullet}$ be bounded complexes of injective objects in $\textnormal{QCoh}(X)$. Then $\textnormal{Tot}(V^{\bullet}\otimes I^{\bullet})$ is a bounded complex of injectives objects of $\textnormal{QCoh}(X)$, and 
			$$\textnormal{Tot}(V^{\bullet}\otimes\textnormal{Hom}^{\bullet}(G^{\bullet},I^{\bullet}))\to\textnormal{Hom}^{\bullet}(G^{\bullet},\textnormal{Tot}(V^{\bullet}\otimes I^{\bullet})) $$
			is a quasi-isomorphism.
		\end{lemma}	
		\begin{proof}
			A direct sum of injective quasi-coherent modules on a Noetherian scheme is an injective quasi-coherent module (Corollary B.2 of \cite{_oupek_2019}). Thus, $\text{Tot}(V^{\bullet}\otimes I^{\bullet})$ and $\text{Tot}(H^{\bullet}(V^{\bullet})\otimes I^{\bullet})$ are quasi-isomorphic bounded-below complexes of injectives, and so $\text{Hom}(G^{\bullet},\text{Tot}(V^{\bullet}\otimes I^{\bullet}))$ is quasi-isomorphic to $\text{Hom}(G^{\bullet},\text{Tot}(H^{\bullet}(V^{\bullet})\otimes I^{\bullet}))$. Finally, since tensoring with finite dimensional vector spaces commute with taking $\text{Hom}$ and totalization, the result follows. 
		\end{proof}
		Now we can construct the projection complex and spectral sequence. 
		\begin{proposition}\label{projection-complex}
			Let $X$ be a variety with $\mathcal{B}=\langle\mathcal{E}_{1},...,\mathcal{E}_{n}\rangle\subset D^{b}_{\textnormal{Coh}}(X)$ an exceptional collection, $i:\mathcal{B}^{\perp}\hookrightarrow D^{b}_{\textnormal{Coh}}(X)$ the inclusion, and $K\in D^{b}_{\textnormal{Coh}}(X)$. Let $\mathcal{E}_{i}^{\bullet}$ and $K^{\bullet}$ be bounded complexes of injective objects of $\textnormal{QCoh}(X)$ quasi-isomorphic to $\mathcal{E}_{i}$ and $K$ respectively. Then the projection $i^{*}K$ is quasi-isomorphic to $\textnormal{Tot}(P^{\bullet,\bullet})$, where $P^{0,\bullet}=K^{\bullet}$ and $P^{-1-p,\bullet}$ for $p\geq0$ is 
			$$\bigoplus_{0\leq a_{0}<...<a_{p}\leq n}\textnormal{Tot}\left(\mathcal{E}_{a_{0}}^{\bullet}\otimes\textnormal{Hom}^{\bullet}(\mathcal{E}_{a_{0}}^{\bullet},\mathcal{E}_{a_{1}}^{\bullet})\otimes...\otimes\textnormal{Hom}^{\bullet}(\mathcal{E}_{a_{p-1}}^{\bullet},\mathcal{E}_{a_{p}}^{\bullet})\otimes\textnormal{Hom}^{\bullet}(\mathcal{E}_{a_{p}}^{\bullet},K^{\bullet})\right)$$
			with maps induced up to a sign by evaluation and composition.
		\end{proposition}
		\begin{proof}
			More precisely, we will show $P^{-1-p,q}$ is
			$$\bigoplus_{\substack{0\leq a_{0}<...<a_{p}\leq n,\\ k_{0}+...+k_{p}+k=q}}\mathcal{E}_{a_{0}}^{k}\otimes\textnormal{Hom}^{k_{0}}(\mathcal{E}_{a_{0}}^{\bullet},\mathcal{E}_{a_{1}}^{\bullet})\otimes...\otimes\textnormal{Hom}^{k_{a-1}}(\mathcal{E}_{a_{p-1}}^{\bullet},\mathcal{E}_{a_{p}}^{\bullet})\otimes\textnormal{Hom}^{k_{p}}(\mathcal{E}_{a_{p}}^{\bullet},K^{\bullet})$$
			with maps $P^{-1-p,q}\to P^{-p,q}$ induced by the evaluation map
			$$\mathcal{E}_{a_{0}}^{k}\otimes\text{Hom}^{k_{0}}(\mathcal{E}_{a_{0}}^{\bullet},\mathcal{E}_{a_{1}}^{\bullet})\to \mathcal{E}_{a_{1}}^{k+k_{0}},\quad
			e_{a_{0}}\otimes\varphi_{a_{0},a_{1}}\mapsto \varphi_{a_{0},a_{1}}(e_{a_{0}})$$
			and the signed composition map
			$$\text{Hom}^{k_{i}}(\mathcal{E}_{a_{i}}^{\bullet},\mathcal{E}_{a_{i+1}}^{\bullet})\otimes
			\text{Hom}^{k_{i+1}}(\mathcal{E}_{a_{i+1}}^{\bullet},\mathcal{E}_{a_{i+2}}^{\bullet})\to
			\text{Hom}^{k_{i}+k_{i+1}}(\mathcal{E}_{a_{i}}^{\bullet},\mathcal{E}_{a_{i+2}}^{\bullet}),$$
			$$\varphi_{a_{i},a_{i+1}}\otimes\varphi_{a_{i+1},a_{i+2}}\mapsto (-1)^{i+1}\varphi_{a_{i+1},a_{i+2}}\circ\varphi_{a_{i},a_{i+1}} $$
			for maps $\varphi_{a_{i},a_{i+1}}$ and local sections $e_{a_{0}}$. These can be seen to compose to zero.			
			
			To prove the statement we induct on the number of exceptional objects in the collection. The base case follows from the standard projection exact triangle 
			$$\text{Tot}(\mathcal{E}_{1}^{\bullet}\otimes\text{Hom}^{\bullet}(\mathcal{E}_{1}^{\bullet},K^{\bullet}))\to K^{\bullet}\to i^{*}K$$
			noting that the mapping cone of a map of complexes coincides with the totalization as a double complex.
			
			For the inductive step, let $\mathcal{B}'=\langle\mathcal{E}_{0},...,\mathcal{E}_{n}\rangle$ be a larger exceptional collection, and let $i_{0}:\mathcal{B}'^{\perp}\hookrightarrow\mathcal{B}^{\perp}$. Thus, $i\circ i_{0}=i':\mathcal{B}'^{\perp}\hookrightarrow D^{b}_{\text{Coh}}(X)$, and $i'^{*}=i_{0}^{*}\circ i^{*}$. Then assuming $i^{*}K$ can be represented by $\text{Tot}(P^{\bullet,\bullet})$ as above, we will show the analogous claim for $i_{0}^{*}i^{*}K$. We have 
			$$i_{0}^{*}i^{*}K =\text{Tot}(\text{Tot}(\mathcal{E}_{0}^{\bullet}\otimes\text{Hom}^{\bullet}(\mathcal{E}_{0}^{\bullet},\text{Tot}(P^{\bullet,\bullet})))\to\text{Tot}(P^{\bullet,\bullet})).$$
			This complex has the same terms as $\text{Tot}(P'^{\bullet,\bullet})$ with $P'^{-1-p,q}$
			\begin{align*}
				&=
				\bigoplus_{k}\mathcal{E}_{0}^{k}\otimes\text{Hom}^{q-k}(\mathcal{E}_{0}^{\bullet},P^{-p,\bullet})\oplus P^{-1-p,q} \\&=
				\bigoplus_{\substack{0\leq a_{0}<...<a_{p-1}\leq n,\\ k_{0}'+k_{0}+...+k_{p-1}+k=q}}\mathcal{E}_{0}^{k}\otimes\textnormal{Hom}^{k_{0}'}(\mathcal{E}_{0}^{\bullet},\mathcal{E}_{a_{0}}^{\bullet})\otimes\text{Hom}^{k_{0}}(\mathcal{E}_{a_{0}}^{\bullet},\mathcal{E}_{a_{1}}^{\bullet})\otimes...\otimes\textnormal{Hom}^{k_{p-1}}(\mathcal{E}_{a_{p-1}}^{\bullet},K^{\bullet})
				\oplus P^{-1-p,q},
			\end{align*}
			where we use Lemma \ref{pullout} to commute $\text{Hom}^{q-k}(\mathcal{E}^{\bullet},-)$ with the tensor product in $\mathcal{E}_{a_{0}}^{\bullet}\otimes...$.  The last line matches the terms of the desired total complex: the first summand consists of terms containing $\mathcal{E}_{0}$, and the latter consists of terms that do not. 
			
			Finally, we must check that the maps induced by the totalization are as described above; note that the outermost totalization corresponds to the mapping cone. The maps within the latter summand are the same by the inductive hypothesis. The maps from the first summand to the latter are induced by evaluation from the projection triangle defining $i_{0}^{*}(i^{*}K)$. And the maps within the former summand are induced by composition, where the desired signs arise since they are flipped in the source when taking a mapping cone.
		\end{proof}
		\begin{corollary}\label{spectral-sequence}
			There is a spectral sequence converging to $\textnormal{Hom}^{p+q}(i^{*}K',i^{*}K)$ with $E_{1}^{-p-1,q}$ equal to
			$$\bigoplus_{\substack{0\leq a_{0}<...<a_{p}\leq n,\\ k_{0}+...+k_{p}+k=q}}\textnormal{Hom}^{k}(K',\mathcal{E}_{a_{0}})\otimes\textnormal{Hom}^{k_{0}}(\mathcal{E}_{a_{0}},\mathcal{E}_{a_{1}})\otimes...\otimes\textnormal{Hom}^{k_{p-1}}(\mathcal{E}_{a_{p-1}},\mathcal{E}_{a_{p}})\otimes\textnormal{Hom}^{k_{p}}(\mathcal{E}_{a_{p}},K)$$
			
			for $p\geq0$ and $E_{1}^{0,q}=\textnormal{Hom}^{q}(K',K)$, with $d_{1}$ given by signed composition. 
		\end{corollary}
		\begin{proof}
			Note $\text{Hom}^{p+q}(i^{*}K',i^{*}K)=\text{Hom}^{p+q}(K',i^{*}K)$, and let $K'^{\bullet}$ be a bounded complex of quasi-coherent modules be an injective complex of quasi-coherent sheaves quasi-isomorphic to $K'$. Then by Proposition \ref{projection-complex}, $\text{Hom}^{p+q}(K',i^{*}K)$ is the cohomology of $\text{Tot}(H^{\bullet,\bullet})$, with $H^{0,\bullet}=\text{Hom}^{\bullet}(K'^{\bullet},K^{\bullet})$ and $H^{-p-1,\bullet}=$
			$$\bigoplus_{0\leq a_{0}<...<a_{p}\leq n}\text{Tot}\left(\textnormal{Hom}^{\bullet}(K'^{\bullet},\mathcal{E}_{a_{0}}^{\bullet})\otimes\textnormal{Hom}^{\bullet}(\mathcal{E}_{a_{0}}^{\bullet},\mathcal{E}_{a_{1}}^{\bullet})\otimes...\otimes\textnormal{Hom}^{\bullet}(\mathcal{E}_{a_{p-1}}^{\bullet},\mathcal{E}_{a_{p}}^{\bullet})\otimes\textnormal{Hom}^{\bullet}(\mathcal{E}_{a_{p}}^{\bullet},K^{\bullet})\right)$$
			with differentials given by signed composition.
			
			Taking the spectral sequence computing the cohomology of the totalization of this double complex (see \cite{stacks-project} Tag \href{https://stacks.math.columbia.edu/tag/012X}{012X}), the differential of the $E_{0}$ is induced by the differentials of the injective complexes, so turning to the $E_{1}$ page computes the derived $\text{Hom}$'s, yielding the desired terms on the $E_{1}$ page, with differential induced by the horizontal maps of the double complex. 
		\end{proof}
		
		\begin{remark}
			There is also an exact triangle for the right adjoint of $i:\langle\mathcal{E}\rangle^{\perp}\hookrightarrow D^{b}_{\text{Coh}}(X)$ for an exceptional object $\mathcal{E}$:
			$$i^{!}K=\text{Cone}(K\to\text{Hom}(K,\mathcal{E}\otimes\omega_{X})^{\vee}\otimes \mathcal{E}\otimes\omega_{X})[-1]. $$
			We can again iterate this process for an exceptional collection $i:\langle\mathcal{E}_{1},...,\mathcal{E}_{n}\rangle\hookrightarrow D^{b}_{\text{Coh}}(X)$ to get a double complex whose totalization is quasi-isomorphic to $i^{!}K$: 
			$$K^{\bullet}\to
			\bigoplus_{i}\text{Hom}^{\bullet}(K^{\bullet},\mathcal{E}_{i}^{\bullet}\otimes\omega_{X})^{\vee}\otimes\mathcal{E}_{i}^{\bullet}\otimes\omega_{X} \to
			\bigoplus_{i<j}\text{Hom}^{\bullet}(K^{\bullet},\mathcal{E}_{i}^{\bullet}\otimes\omega_{X})^{\vee}\otimes\text{Hom}^{\bullet}(\mathcal{E}_{i}^{\bullet},\mathcal{E}_{j}^{\bullet})^{\vee}\otimes\mathcal{E}_{j}^{\bullet}\otimes\omega_{X}\to... $$
%			using the fact that the cohomology of the dual of a complex of vector spaces is the dual of the cohomology. 
			And we get a spectral sequence computing $\text{Hom}^{*}(i^{!}K,i^{!}K')$ with $E_{1}^{0,q}=\text{Hom}^{q}(K,K')$ and $E_{1}^{p+1,q}$ for $p\geq0$ equal to
			$$\bigoplus_{\substack{0\leq a_{0}<...<a_{p}\leq n,\\ k_{0}+...+k_{p}+k=q}}\text{Hom}^{k_{0}}(K,\mathcal{E}_{a_{0}}\otimes\omega)\otimes\text{Hom}^{k_{1}}(\mathcal{E}_{a_{0}},\mathcal{E}_{a_{1}})\otimes...\otimes\text{Hom}^{k_{p-1}}(\mathcal{E}_{a_{p-1}},\mathcal{E}_{a_{p}})\otimes\text{Hom}^{k_{p}}(\mathcal{E}_{a_{p}}\otimes\omega,K'). $$
		\end{remark}
	\section{A Strong Generator}\label{section-generator}
		The first object we will construct using the projection $i^{*}$ will be a strong generator for $\mathcal{A}$. An interesting result will follow from an analysis of negative self-Homs, towards which we show the following:		
		\begin{lemma}\label{negative-terms}
			Let $\mathcal{F},\mathcal{F}'\in\textnormal{Coh}(X)$. Then $\textnormal{Hom}^{*}(i^{*}\mathcal{F}',i^{*}\mathcal{F})$ is zero for $*<-1$, and 
			$$\textnormal{Hom}^{-1}(i^{*}\mathcal{F}',i^{*}\mathcal{F})=\ker\left(\bigoplus_{i=1}^{13}\textnormal{Hom}^{0}(\mathcal{F}',\mathcal{E}_{i})\otimes\textnormal{Hom}^{0}(\mathcal{E}_{i},\mathcal{F})\to \textnormal{Hom}^{0}(\mathcal{F}',\mathcal{F})\right).$$
		\end{lemma}
		\begin{proof}
			Consider the $E_{1}$ page of the spectral sequence of Corollary \ref{spectral-sequence}. Note that all nontrivial Homs between sheaves have non-negative degree and by Lemma \ref{vanishing}, $\text{Hom}^{*}(\mathcal{E}_{i},\mathcal{E}_{j})$ is zero except in degree $2$. Thus, $E_{1}^{p,q}$ is zero for $p>0$, for $p=0,q<0$, for $p=-1,q<0$, and for $p<-1,q<-2p-2$. The only potentially nonzero term with negative total degree is $E_{1}^{-1,0}$, and for degree reasons the only non-zero differential into or out of $E_{r}^{-1,0}$ is $d_{1}:E_{1}^{-1,0}\to E_{1}^{0,0}$, yielding the desired kernel.
		\end{proof}
		
		Now, consider the full exceptional collection:
		$$\langle\mathcal{L}_{1},...,\mathcal{L}_{13}\rangle:=\langle\mathcal{O}_{X},\mathcal{O}_{X}(E_{1}),...,\mathcal{O}_{X}(E_{10}),\mathcal{O}_{X}(F),\mathcal{O}_{X}(2F)\rangle. $$
		
		\begin{lemma}
			$\textnormal{Hom}^{*}(i^{*}\mathcal{L}_{i},i^{*}\mathcal{L}_{j})$ is $0$ for $*<0$.
		\end{lemma}
		\begin{proof}
			By Lemma \ref{negative-terms}, it suffices to examine $d_{1}:E_{1}^{-1,0}\to E_{1}^{0,0}$. If $\mathcal{E}_{i}$ and $\mathcal{L}_{j}$ are distinct sheaves, then $\text{Hom}^{0}(\mathcal{L}_{i},\mathcal{E}_{j})\otimes\text{Hom}^{0}(\mathcal{E}_{j},\mathcal{L}_{i})=0$. The only other case is $\mathcal{L}_{i}=\mathcal{L}_{j}=\mathcal{O}_{X}$, in which case $d_{1}$ is the composition map
			$$\text{Hom}^{0}(\mathcal{O}_{X},\mathcal{O}_{X})\otimes\text{Hom}^{0}(\mathcal{O}_{X},\mathcal{O}_{X})\to\text{Hom}^{0}(\mathcal{O}_{X},\mathcal{O}_{X})$$
			which has trivial kernel.
		\end{proof}
		Now, let $\mathcal{Q}=\bigoplus_{i=1}^{13}i^{*}\mathcal{L}_{i}$.
		\begin{proposition}
			$\mathcal{Q}$ is a strong generator for the phantom with no negative self-Homs. 
		\end{proposition}
		\begin{proof}
			Every element of the phantom can be written as $i^{*}K$ for $K\in D^{b}_{\text{Coh}}(X)$, each $K$ can be written in terms of direct sums, cones, and summands of the $\mathcal{L}_{i}$'s, and all of these operations commute with $i^{*}$.
		\end{proof}
		\begin{remark}
			In fact, a strong generator of a phantom $\mathcal{A}$ generates it with direct sums and cones and no need for summands. This fact follows from Theorem 2.1 of \cite{Thomason1997}. Indeed, if the triangulated subcategory generated without summands were not all of $\mathcal{A}$, it would be a proper dense subcategory and therefore correspond to a proper subgroup $K_{0}(\mathcal{A})=0$, which is impossible.
		\end{remark}
		This example yields an affirmative answer to a question of Ben Antieau:		
		\begin{theorem}\label{Antieau}
			There exists a co-connective dg-algebra whose derived category is a phantom.
		\end{theorem}
		\begin{proof}
			In a dg-enhancement of $\mathcal{A}$, consider the Ext-algebra of $\mathcal{Q}$. Since all negative self-Exts of $\mathcal{Q}$ are zero, this will be a co-connective dg-algebra. And since $\mathcal{Q}$ is a strong generator for $\mathcal{A}$, the derived category of this dg-algebra will be $\mathcal{A}$.
		\end{proof} 
		
	\section{Projected Skyscraper Sheaves}\label{section-skyscraper}
		Let $\kappa(x)$ be the sky-scraper sheaf at a closed point $x\in X$.
	\subsection{Homs}		
		We will compute the graded vector spaces $\text{Hom}^{*}(i^{*}\kappa(x),i^{*}\kappa(x))$ and $\text{Hom}^{*}(i^{*}\kappa(x),i^{*}\kappa(y))$. 
		\begin{lemma}\label{Euler-characteristics}
			Denoting $\chi(D):=\chi(H^{*}(X,\mathcal{O}_{X}(D)))$,
			$$\chi(D_{i})=1,\quad
			\chi(F)=3, \quad
			\chi(2F)=6,\quad \chi(F-D_{i})=2, \quad
			\chi(2F-D_{i})=5.$$
		\end{lemma}
		\begin{proof}
			Can be checked using the Riemann-Roch Theorem for surfaces, and the computations can be simplified by noting that the reflection $F\mapsto H, D_{i}\mapsto E_{i}$ preserves Euler characteristics.
		\end{proof}
		\begin{proposition}\label{Hom(k(x),k(x))}
			For any $x\in X$,
			$$\textnormal{Hom}^{*}(i^{*}\kappa(x),i^{*}\kappa(x))=\mathbf{C}^{1}[0]\oplus\mathbf{C}^{14}[-1]\oplus\mathbf{C}^{92}[-2]\oplus\mathbf{C}^{139}[-3]\oplus\mathbf{C}^{60}[-4].$$
		\end{proposition}
		
		\begin{proof}
			Let us identify the terms on the $E_{1}$ page of the spectral sequence of Corollary \ref{spectral-sequence}. 
			
			The column $E_{1}^{0,q}$ consists of $\text{Hom}^{*}(\kappa(x),\kappa(x))=\mathbf{C}[0]\oplus\mathbf{C}^{2}[-1]\oplus\mathbf{C}[-2]$, as can be checked using the Koszul resolution of $\kappa(x)$. 
			
			Next, note that $\textnormal{Hom}^{*}(\mathcal{E}_{i},\kappa(x))=\mathbf{C}[0]$ and $\text{Hom}^{*}(\kappa(y),E_{i})=\mathbf{C}[-2]$, which along with Lemma \ref{vanishing} means the expression
			$$\textnormal{Hom}^{k}(\kappa(x),\mathcal{E}_{a_{0}})\otimes\textnormal{Hom}^{k_{0}}(\mathcal{E}_{i_{0}},\mathcal{E}_{a_{1}})\otimes...\otimes\textnormal{Hom}^{k_{p-1}}(\mathcal{E}_{a_{p-1}},\mathcal{E}_{a_{p}})\otimes\textnormal{Hom}^{k_{p}}(\mathcal{E}_{a_{p}},\kappa(x))$$
			is zero unless at most one of $\mathcal{E}_{a_{j}}$ is of the form $\mathcal{O}_{X}(-D_{i})$ and $k=2$, $k_{p}=0$, and $k_{a_{j}}=2$ for $0\leq a_{j}\leq p-1$. In particular, $E_{1}^{p,q}$ is nonzero if and only if $-p\leq 4$ and $q=2p$. Denoting
			$$\rho(\mathcal{E}_{a_{0}},..,E_{a_{p}})=\sum_{j=0}^{p-1}\chi(E_{a_{j+1}})-\chi(E_{a_{j}}), $$
			using Lemma \ref{Euler-characteristics} we compute 
			\begin{align*}
				\dim E_{1}^{-1,2} &= 13 \\
				\dim E_{1}^{-2,4} &=
				10\rho(\mathcal{O}(-D_{i}),\mathcal{O})+
				\rho(\mathcal{O}(-F),\mathcal{O})+
				\rho(\mathcal{O}(-2F),\mathcal{O}) \\&+
				10\rho(\mathcal{O}(-F),\mathcal{O}(-D_{i}))+10\rho(\mathcal{O}(-2F),\mathcal{O}(-D_{i}))+\rho(\mathcal{O}(-2F),\mathcal{O}(-F)) \\&=
				10+3+6+10\cdot 2+10\cdot 5+3 \\&=
				92\\
				\dim E_{1}^{-3,6} &=
				10\rho(\mathcal{O}(-F),\mathcal{O}(-D_{i}),\mathcal{O})+
				10\rho(\mathcal{O}(-2F),\mathcal{O}(-D_{i}),\mathcal{O})\\&+
				\rho(\mathcal{O}(-2F),\mathcal{O}(-F),\mathcal{O})+10\rho(\mathcal{O}(-2F),\mathcal{O}(-F),\mathcal{O}(-D_{i}))\\&=
				10(1\cdot2)+10(1\cdot5)+(3\cdot3)+10(2\cdot3) \\&=
				139\\
				\dim E_{1}^{-4,8} &= 
				10\rho(\mathcal{O}(-2F),\mathcal{O}(-F),\mathcal{O}(-D_{i}),\mathcal{O}) \\&=
				10(1\cdot2\cdot3) \\&=60
			\end{align*}			
			All differentials vanish for degree reasons except for $d_{1}:E_{1}^{-1,2}\to E_{1}^{0,2}$, which is the sum of evaluation maps
			$$\bigoplus_{i=1}^{13}\text{Hom}^{2}(\kappa(x),\mathcal{E}_{i})\otimes\text{Hom}^{0}(\mathcal{E}_{i},\kappa(x))\to\text{Hom}^{2}(\kappa(x),\kappa(x)). $$
			Since the codomain is $\mathbf{C}$, if any of the evaluation maps is nonzero then the whole map is surjective. Indeed, for any element of $\text{Hom}^{0}(\mathcal{E}_{i},\kappa(x))$, which corresponds to an element of the stalk of $\mathcal{E}_{i}$ at $x$, the evaluation map is 
			$$\text{Hom}^{2}(\kappa(x),\mathcal{E}_{i})\to\text{Hom}^{2}(\kappa(x),\kappa(x)), $$ 
			which by Serre duality is dual to 
			$$\text{Hom}^{0}(\mathcal{E}_{i},\kappa(x))\leftarrow\text{Hom}^{0}(\kappa(x),\kappa(x)), $$
			(noting $\kappa(x)\otimes\omega_{X}\cong\kappa(x)$) which is nonzero as it corresponds to scaling the chosen section of $(\mathcal{E}_{i})_{x}$.
			
			Thus, $\dim E_{2}^{-1,2}=12$, $\dim E_{2}^{0,2}=0$, and all other terms are the same as on $E_{1}$, yielding the desired total degrees for the $E_{\infty}$ page.
		\end{proof}
		% https://tikzcd.yichuanshen.de/#N4Igdg9gJgpgziAXAbVABwnAlgFyxMJZAFgBoAOAXVJADcBDAGwFcYkQAdDnGADx2AAJCAFsAvgD1gABjEAKLgGt6aNPTm8AlKSUq1GzZpBjS6TLnyEUZAOzU6TVuy49+Q0ZOABGebtXqtHQ5lfwMjEzNsPAIiMgA2ewYWNkRObj4BYXEpACZfYL0A7T99LXDTEAwoyyIAZlIEmiSnVK4AIywAcwg0FjgAfWB6QdkxFwz3bOA8hQLQwK4RehwACwBjJmAAUTFB4ZkxMU0uCDwReHG3LM9ZWaXVjcZt3aGRw6CQ0sNjCqqLGJQOVIxESjhSaQ63V6zAGrwOAB59j4xukrh5cvlPkUgvd1psdns3kcTmcLqjMujpvlcY9noSDiZFss8U8CXDkccOKcsOc4JcKVNbkyHviXkj3iUiuVIv8rMgvKQcqDks4OJCen16bJEYMfDqqSjXALPDNJQYcczaWz9qNOdzefzJibqZbRVqJRwaW72Yc7aS+eSnRi7q7WWLdR6vWH6XliVz-Y7rlIhZ7Q3S4bGPoUwj8ZdE5dJSNJlS0IV0NTD3fq9fs8vraodE5TTXMvhaRdG4baSTyyUag1SQx30zbI2nrRG4-a+xMk4PhSyR5P24uJwa-b2A-25y2o0uDSureHgA2pwnA3OU3u16es-NvmJ7DAoJ14ERQAAzABOoiQ9RAOAQEgeSUGIQA
		
		\begin{corollary}\label{Hom1}
			$\textnormal{Hom}^{1}(i^{*}\kappa(x),i^{*}\kappa(x))$ is
			$$\textnormal{Hom}^{1}(\kappa(x),\kappa(x))\oplus\ker\left(\bigoplus_{i=1}^{13}\textnormal{Hom}^{2}(\kappa(x),\mathcal{E}_{i})\otimes\textnormal{Hom}^{0}(\mathcal{E}_{i},\kappa(x))\xrightarrow{\textnormal{ev}}\textnormal{Hom}^{2}(\kappa(x),\kappa(x))\right). $$
		\end{corollary}
		\begin{corollary}\label{different-points}
			For any $x,y\in X$ with $x\neq y$,
			$$\textnormal{Hom}^{*}(i^{*}\kappa(x),i^{*}\kappa(y))=\mathbf{C}^{13}[-1]\oplus\mathbf{C}^{92}[-2]\oplus\mathbf{C}^{139}[-3]\oplus\mathbf{C}^{60}[-4].$$
		\end{corollary}
		\begin{proof}
			The $E_{1}$ page of the spectral sequence of Corollary \ref{spectral-sequence} is the same as in Proposition \ref{Hom(k(x),k(x))}, except $E_{1}^{0,q}=0$ as $\text{Hom}^{*}(\kappa(x),\kappa(y))=0$, so all differentials vanish for degree reasons, and the spectral sequence converges to the desired result.
		\end{proof}
	
	\subsection{Deformations of generic $i^{*}\kappa(x)$}
		The map $p:x\mapsto i^{*}\kappa(x)$ is a morphism from $X$ to the moduli stack of objects in the phantom. Corollary \ref{different-points} shows that it is injective. We now study the image of this map.
	
		Since $\text{Hom}^{*}(i^{*}\kappa(x),i^{*}\kappa(x)$ is zero for $*<0$, we know that $p$ maps into the locus of universally gluable objects (see \cite{Lieblich2006}), which forms an algebraic stack. Moreover, since $\text{Hom}^{0}(i^{*}\kappa(x),i^{*}\kappa(x))=\mathbf{C}$, it maps into the locus of simple objects, which we denote $s\mathcal{M}(\mathcal{A})$, and which is a $\mathbf{G}_{m}$-gerbe. 
		
		Since $\mathcal{A}$ embeds in $D^{b}_{\text{Coh}}(X)$ fully faithfully, $s\mathcal{M}(\mathcal{A})\subset s\mathcal{M}(D^{b}_{\text{Coh}}(X))$ is an open substack, so deformations of objects in $\mathcal{A}$ are still objects in $\mathcal{A}$. However, objects of $\mathcal{A}$ could specialize to objects of $D^{b}_{\text{Coh}}(X)$ not in $\mathcal{A}$.
		
		Since $\text{Hom}^{*}(i^{*}\kappa(x),i^{*}\kappa(x))=\mathbf{C}^{14}$, every $i^{*}\kappa(x)$ has a $14$-dimensional space of infinitesimal deformations. But we will see that many of these deformations are obstructed in a way that depends on the choice of $x$. Studying the hull for deformations will allow us to understand a component of the reduced coarse moduli space of $s\mathcal{M}(\mathcal{A})$. 
		  
		To understand obstruction of deformations, we will need to study composition of $\text{Hom}^{1}$'s. Recalling Corollary \ref{Hom1}, we set the following notation.
		
		Fixing $x\in X$, let $\{s_{1},s_{2}\}$ be a basis for $\text{Hom}^{1}(\kappa(x),\kappa(x))$, let $\xi_{i}$ be a basis for the ($1$-dimensional) space $\text{Hom}^{0}(\mathcal{E}_{i},\kappa(x))\otimes\text{Hom}^{2}(\kappa(x),\mathcal{E}_{i})$ such that all $\xi_{i}$'s have the same image under the composition map to $\text{Hom}^{2}(\kappa(x),\kappa(x))$. Letting $b_{i}=\xi_{i}-\xi_{13}$,
		$$\{s_{1},s_{2},b_{1},...,b_{12}\} $$
		is a basis for $\text{Hom}^{1}(i^{*}\kappa(x),i^{*}\kappa(x))$.
		\begin{lemma}\label{composition}
			The composition map 
			$$\textnormal{Hom}^{0}(\mathcal{E}_{i},\kappa(x))\otimes\textnormal{Hom}^{2}(\kappa(x),\mathcal{E}_{j})\to
			\textnormal{Hom}^{2}(\mathcal{E}_{i},\mathcal{E}_{j}) $$
			is nonzero if and only $\mathcal{E}_{j}^{\vee}\otimes\mathcal{E}_{i}\otimes\omega_{X}$ has global sections and $x$ lies outside the base locus.
		\end{lemma}
		\begin{proof}
			Since $\text{Hom}^{0}(\mathcal{E}_{i},\kappa(x))\cong\mathbf{C}$, this map is the same as 
			$$\text{Hom}^{2}(\kappa(x),\mathcal{E}_{j})\xrightarrow{-\circ\varphi}
			\text{Hom}^{2}(\mathcal{E}_{i},\mathcal{E}_{j})$$
			for any nonzero $\varphi\in\text{Hom}^{0}(\mathcal{E}_{i},\kappa(x))$. By Serre duality, this map is dual to
			$$H^{0}(X,\kappa(x))\leftarrow H^{0}(X,\mathcal{E}_{j}^{\vee}\otimes\mathcal{E}_{i}\otimes\omega_{X})$$
			which corresponds to taking the value of a global section of $\mathcal{E}_{j}^{\vee}\otimes\mathcal{E}_{i}\otimes\omega_{X}$ at $x$, and so is nonzero if and only if there is a global section non-vanishing at $x$. 
		\end{proof}
		Now we can understand the product structure on $\text{Hom}^{1}(i^{*}\kappa(x),i^{*}\kappa(x))$. 
		\begin{lemma}\label{general-compositions}
			For all $x\in X$, we have $s_{i}b_{j}=b_{j}s_{i}=0$ in $\textnormal{Hom}^{2}(i^{*}\kappa(x),i^{*}\kappa(x))$, and for $x$ outside the base locus of all $\mathcal{E}_{j}^{\vee}\otimes\mathcal{E}_{i}\otimes\omega_{X}$, we have $\{b_{i}^{2},b_{i}b_{j}\}_{i<j}$
			is linearly independent and $b_{j}b_{i}=b_{j}^{2}$ for $i<j$.
		\end{lemma}
		\begin{proof}
			The compositions $s_{i}\xi_{j}$ and $\xi_{j}s_{i}$ are 
			$$\text{Hom}^{1}(\kappa(x),\kappa(x))\otimes\text{Hom}^{2}(\kappa(x),\mathcal{E}_{j})\otimes\text{Hom}^{0}(\mathcal{E}_{j},\kappa(x))\to\text{Hom}^{3}(\kappa(x),\mathcal{E}_{j})\otimes\text{Hom}^{0}(\mathcal{E}_{j},\kappa(x)) $$
			and
			$$\text{Hom}^{2}(\kappa(x),\mathcal{E}_{j})\otimes\text{Hom}^{0}(\mathcal{E}_{j},\kappa(x))\otimes\text{Hom}^{1}(\kappa(x),\kappa(x))\to\text{Hom}^{2}(\kappa(x),\mathcal{E}_{j})\otimes\text{Hom}^{1}(\mathcal{E}_{j},\kappa(x))$$
			respectively, and both have vanishing codomain, which implies the first part.
			
			The composition $\xi_{i}\xi_{j}$ is 
			\begin{align*}
				&\text{Hom}^{2}(\kappa(x),\mathcal{E}_{i})\otimes\text{Hom}^{0}(\mathcal{E}_{i},\kappa(x))\otimes\text{Hom}^{2}(\kappa(x),\mathcal{E}_{j})\otimes\text{Hom}^{0}(\mathcal{E}_{j},\kappa(x))\\&\to
				\text{Hom}^{2}(\kappa(x),\mathcal{E}_{i})\otimes\text{Hom}^{2}(\mathcal{E}_{i},\mathcal{E}_{j})\otimes\text{Hom}^{0}(\mathcal{E}_{j},\kappa(x)) 
			\end{align*}
			induced by the map of Lemma \ref{composition}, which by that lemma is full rank for $i<j$ assuming the hypothesis on $x$, and is zero if $i\geq j$. And by the proof of Proposition \ref{Hom(k(x),k(x))}, for each $i<j$ the codomain is an independent subspace of $\text{Hom}^{2}(i^{*}\kappa(x),i^{*}\kappa(x))$, so $\{\xi_{i}\xi_{j}\}_{i<j}$ is linearly independent. Thus, for $i<j$, we can compute
			\begin{align*}
				b_{i}^{2} &=
				(\xi_{i}-\xi_{13})^{2} \\&=
				\xi_{i}^{2}-\xi_{i}\xi_{13}-\xi_{13}\xi_{i}+\xi_{0}^{2} \\&
				=-\xi_{i}\xi_{13} \\
				b_{i}b_{j} &=
				(\xi_{i}-\xi_{13})(\xi_{j}-\xi_{13}) \\&=
				\xi_{i}\xi_{j}-\xi_{i}\xi_{13}-\xi_{13}\xi_{j}+\xi_{13}^{2} \\&=
				\xi_{i}\xi_{j}-\xi_{i}\xi_{13}
			\end{align*} 
			for which linear independence follows. And we have 
			\begin{align*}
				b_{j}b_{i} &=
				(\xi_{j}-\xi_{13})(\xi_{i}-\xi_{13}) \\&=
				\xi_{j}\xi_{i}-\xi_{j}\xi_{13}-\xi_{13}\xi_{i}+\xi_{13}^{2} \\&=
				-\xi_{j}\xi_{13} \\&=
				b_{j}^{2}.
			\end{align*}
		\end{proof}
		
		Now we turn to the deformation theory of $i^{*}\kappa(x)$. Since $s\mathcal{M}(\mathcal{A})$ is an algebraic stack of finite type over $\mathbf{C}$, for any object $\alpha\in\mathcal{A}$ the deformation functor $F$ of $\alpha$ admits a \emph{hull} for deformations: there exists $R\cong\mathbf{C}[[z_{1},...,z_{n}]]/I$ such that the pro-representable functor $h_{R}$ has a smooth map $h_{R}\to F$ that is an isomorphism on tangent spaces. And then we have $R\cong\widehat{\mathcal{O}}_{sM(\mathcal{A}),[\alpha]}$, where $sM(\mathcal{A})$ is the coarse space of $s\mathcal{M}(\mathcal{A})$ and $[\alpha]$ is the point corresponding to $\alpha$.  
	
		\begin{lemma}\label{obstructed}
			For generic $x\in X$, the reduction of the hull for deformations of $i^{*}\kappa(x)$ is formally smooth and $2$-dimensional, i.e., $\mathbf{C}[[y_{0},y_{1}]]$.
		\end{lemma}
		\begin{proof}
			Denote $S:=k[[y_{0},y_{1},x_{1},...,x_{12}]]$, where $y_{0}, y_{1}$ are dual to the tangent vectors $s_{1}, s_{2}$, and $x_{1},...,x_{12}$ are dual to the tangent vectors $b_{1},...,b_{12}$. The hull is of the form $R:=S/I$, with $I\subset\mathfrak{m}^{2}$ for $\mathfrak{m}\subset S$ the maximal ideal. We know that deformations in the directions $s_{1},s_{2}$ are unobstructed as they correspond to deforming $x\in X$. We will show that all deformations in the $b_{i}$ directions are obstructed.
			
			Denote $T^{i}:=\text{Hom}^{j}(i^{*}\kappa(x),i^{*}\kappa(x))$, and note that $T^{1}\cong(\mathfrak{m}/\mathfrak{m}^{2})^{\vee}$ is the tangent space. Since $R$ is the hull, there is a deformation over $R$ that cannot be lifted further. In particular, there is a non-trivial obstruction to lifting along the quotient map
			$$S/\mathfrak{m}I\to R $$
			whose kernel is $I/\mathfrak{m}I$. This obstruction is a nonzero element of $T^{2}\otimes I/\mathfrak{m}I$ corresponding to a surjective map $(T^{2})^{\vee}\to I/\mathfrak{m}I$ (though we will not end up needing that it is surjective). 
			
			There is a natural map induced by inclusion $I/\mathfrak{m}I\to\mathfrak{m}^{2}/\mathfrak{m}^{3}$, and we have $\mathfrak{m}^{2}/\mathfrak{m}^{3}\cong\text{Sym}^{2}(\mathfrak{m}/\mathfrak{m}^{2})\cong\text{Sym}^{2}(T^{1})^{\vee}$. It follows from Section 2.3 of \cite{remarks}, in the last paragraph on the quadratic part of the Kuranishi map, that the composition
			$$\begin{tikzcd}
				(T^{2})^{\vee} \arrow[r] \arrow[d] & \text{Sym}^{2}(T^{1})^{\vee}\cong\mathfrak{m}^{2}/\mathfrak{m}^{3} \\
				I/\mathfrak{m}I \arrow[ur]          &            
			\end{tikzcd}$$
			is dual to
			$$\text{Sym}^{2}(T^{1})\to T^{2},\quad tt'\mapsto t\circ t'+t'\circ t,$$
			which by Lemma \ref{general-compositions} sends $b_{i}b_{j}\mapsto b_{j}^{2}+b_{i}b_{j}$ for $i<j$.
			
			Letting $\{\overline{f}_{i},\overline{f}_{i,j}\}_{i<j}$ be basis elements of $(T^{2})^{\vee}$ dual to $\{b_{i}^{2},b_{i}b_{j}\}_{i<j}\subset T^{2}$, their images $f_{i},f_{i,j}\in \mathfrak{m}^{2}/\mathfrak{m}^{3}$ are 
			$$f_{i}=2x_{i}^{2}+\sum_{i<j}x_{i}x_{j}+\mathfrak{m}^{3}, \quad f_{i,j}=x_{i}x_{j}+\mathfrak{m}^{3}.$$   
			Taking the preimages of these elements in $I/\mathfrak{m}I$, it follows that $I/\mathfrak{m}^{3}$ contains all $x_{i}^{2}$ and $x_{i}x_{j}$. 
			
			Let $I':=(I,y_{0},y_{1})\subset S$. We can see $x_{i}^{2}\in I'=(I,y_{0},y_{1})$ by inductively showing $x_{i}^{2}\in I'/\mathfrak{m}^{k}$, noting that the highest order term can be killed off up to higher order terms. Thus, $S/I'\cong R/(y_{0},y_{1})$ is zero dimensional, so by \cite{stacks-project}[00KW] $R$ is at most $2$-dimensional. And as noted above, it is at least $2$-dimensional, so $S/\sqrt{I}$ is a regular complete local ring of dimension $2$.
		\end{proof}
		Now the $p:X\to\mathcal{A}$ recovers the following.
		\begin{proposition}\label{birational}
			Let $sM(\mathcal{A})$ be the reduced coarse moduli space of $s\mathcal{M}(\mathcal{A})$. Then an irreducible component of $sM(\mathcal{A})$ is birational to $X$.
		\end{proposition} 
		\begin{proof}
			By Lemma \ref{obstructed}, for general $x\in X$, the only unobstructed deformations of $i^{*}\kappa(x)$ are those corresponding to deformations of $x\in X$. Let $U$ be the locus of such $x$. Then $p|_{U}$ is an open immersion, and its image is dense in an irreducible component.  
		\end{proof}
		For a special point $x\in X$ to which Lemma \ref{obstructed} does not apply, $i^{*}\kappa(x)$ may have additional deformations not corresponding to deformations of $x$. That would mean the irreducible component of $sM(\mathcal{A})$ containing $p(X)$ intersects other components of $sM(\mathcal{A})$. We ask the following:
		
		\begin{question}\label{mutant}
			Can we construct a simple universally gluable object of $\mathcal{A}$ that deforms to $i^{*}\kappa(x)$ for some $x\in X$ but is not itself of this form?
		\end{question}
		
		Furthermore, $p$ may fail to be closed immersion, as $sM(\mathcal{A})$ may be non-separated, raising the following:  
		\begin{question}
			Can we find a specialization of $i^{*}\kappa(x)$ for $x\in X$ to a simple universally gluable object of $\mathcal{A}$ not of this form? 
		\end{question}

	\subsection{Special loci}
	
		Now we study the points $x\in X$ which could have additional unobstructed deformations. 
		
		\begin{definition}
			Say $x\in X$ lies in the \emph{special locus of $X$} if the rank of the product map
			$$\text{Hom}^{1}(i^{*}\kappa(x),i^{*}\kappa(x))\otimes\text{Hom}^{1}(i^{*}\kappa(x),i^{*}\kappa(x))\to\text{Hom}^{2}(i^{*}\kappa(x),i^{*}\kappa(x))$$
			is less than in the general case. 
		\end{definition}
		\begin{remark}
			Our tools are not sufficient to show for $x$ in the special locus of $X$ that $i^{*}\kappa(x)$ actually has additional deformations, as asked in Question \ref{mutant}.
		\end{remark}
	
%		\begin{lemma}
%			The base loci of $|-K_{X}+E_{i}|$ lie in the special locus of $X$.
%		\end{lemma}
%		\begin{proof}
%			Letting $\mathcal{E}_{j}=\mathcal{O}_{X}(-D_{i})$, we have $\mathcal{O}_{X}(-K_{X}+E_{i})=\mathcal{E}_{13}^{\vee}\otimes\mathcal{E}_{j}\otimes\omega_{X}$, so by the computations in the proof of Lemma \ref{general-compositions}, $x$ in this locus has $b_{i}^{2}=0$. 
%		\end{proof}
				
		\begin{lemma}\label{linear-systems}
			The special locus of $X$ consists of the base loci of the following linear systems: $|-K_{X}+E_{i}|$, $|K_{X}-F|$, $|K_{X}-2F|$, $|K_{X}-F+D_{i}|$, and $|K_{X}-2F+D_{i}|$.
		\end{lemma}	
		\begin{proof}
			These are the points $x\in X$ that fail the hypotheses of Lemma \ref{general-compositions}.
		\end{proof}
		
		We will show that only the first of these linear systems has divisors in its base locus.		
			
		\begin{proposition}\label{elliptic-curves}
			The only irreducible effective divisors supported in the special locus of $X$ are $-K_{X}+E_{i}$ for $1\leq i\leq 10$. 
		\end{proposition}
		\begin{proof}
			We can compute using Riemann-Roch for surfaces that $\chi(-K_{X}+E_{i})=1$, and so by Lemma \ref{standard-form} we indeed have $|-K_{X}+E_{i}|=0$ so the unique effective divisor is the base locus.
			
			Next, for each of the other linear systems $|D|$, assuming $B$ is an effective divisor in the base locus, and so $D-B$ is also effective, we will reach a contradiction.  
			\begin{itemize}
				\item $|D|:=|K_{X}-F|=|16H-5\sum_{j=1}^{10}E_{j}|$.
				
				Since all the coefficients of the $E_{i}$'s are the same, by Lemma \ref{permutation} each $\sigma B$ is in the base locus for $\sigma\in S_{10}$, so we can replace $B$ with $\sum_{\{\sigma B\}}\sigma B$ and assume $B\sim dH-m\sum_{j=1}^{10}E_{j}$. For both $B$ and $D-B$ to be effective, by Lemma \ref{non-special}, we must have $\frac{d}{m},\frac{16-d}{5-m}\geq \frac{2280}{721}$, which we can see is impossible for integers $0\leq d\leq 16$ and $0\leq m\leq 5$.
				
				\item $|D|:=|K_{X}-2F|=|35H-11\sum_{j=1}^{10}E_{j}|$.
				
				By the same argument as above, we must have $\frac{d}{m},\frac{35-d}{11-m}\geq \frac{2280}{721}$. The only possible case is 
				$$D=\left(19H-6\sum_{j=1}^{10}E_{j}\right)+\left(16H-5\sum_{j=1}^{10}E_{j}\right) $$
				but we know $|19H-6\sum_{j=1}^{10}E_{j}|$ is empty as its Euler characteristic is zero and Lemma \ref{standard-form} applies.
				
				\item $|D|:=|K_{X}-F+D_{i}|=|10H-3\sum_{j=1}^{10}E_{j}-E_{i}|$.
				
				Again applying Lemma \ref{permutation}, we can assume $B\sim dH-m\sum_{j=1}^{10}E_{j}-m'E_{i}$. Taking the sum over $i=1,...,10$, we see that an effective divisor in $|100H-31\sum_{j=1}^{10}|$ has a component linearly equivalent to $10dH-(10m+m')\sum_{j=1}^{10}E_{j}$. From the conditions
				$$\frac{10d}{10m+m'},\frac{100-10d}{31-(10m+m')}\geq\frac{2280}{721} $$
				for integers $0\leq d \leq 5$ (assuming without loss of generality $d\leq 10-d$), $0\leq m\leq 3$, $0\leq m'+m\leq 4$, we check that we are left with the following cases, each of which we show is impossible:
				\begin{itemize}
					\item $d=1$, $m=0$, $m'=3$.
					
					$|B|=|H-3E_{i}|$ is empty, as no line intersects a point with multiplicity $3$.
					\item $d=3$, $m=1$, $m'=-1$.
					
					$|D-B|=|7H-2\sum_{j=1}^{10}E_{j}-2E_{i}|$, and we can apply Lemma \ref{Cremona} to take Cremona reductions:
					\begin{align*}
						\dim|D-B| &= \dim|7H-4E_{1}-2E_{2}-2E_{3}-...-2E_{10}| \\&=
						\dim|6H-3E_{1}-2E_{2}-...-2E_{8}-E_{9}-E_{10}| \\&=
						\dim|5H-2E_{1}-...-2E_{6}-E_{7}-E_{8}-E_{9}-E_{10}| \\&=
						\dim|4H-2E_{1}-2E_{2}-2E_{3}-E_{4}-...-E_{10}| \\&=
						\dim|2H-E_{1}-....-E_{7}| \\&=
						\dim|H-E_{1}-E_{2}-E_{3}-E_{4}| 
					\end{align*}
					and this last linear system is empty, as there is no line through $4$ general points in $\mathbf{P}^{2}$.
				\end{itemize}
				\item $|D|:=|K_{X}-2F+D_{i}|=|29H-9\sum_{j=1}^{10}E_{i}-E_{i}|$.
				
				By the same argument as above, we must have 
				$$\frac{10d}{10m+m'},\frac{290-10d}{91-(10m+m')}\geq\frac{2280}{721} $$
				with $0\leq d\leq 14$ (assuming without loss of generality $d\leq 29-d$), $0\leq m\leq 9$, and $0\leq m+m'\leq 10$, leaving the cases:
				\begin{itemize}  
					\item $d=1$, $m=0$, $m'=3$.
					
					Again $|B|=|H-3E_{i}|$ is empty.
					\item $d=2$, $m=0$, $m'=6$.
					
					$|B|=|2H-6E_{i}|$ is empty, as no conic has  multiplicity $6$ at any point.
				
					\item $d=3$, $m=0$, $m'=9$.
					
					$|B|=|3H-9E_{i}|$ is empty as no planar cubic has multiplicity $9$ at any point.					
					
					\item $d=3$, $m=1$, $m'=-1$.
					
					$|D-B|=|26H-8\sum_{j=1}^{10}E_{j}-2E_{i}|$ has Euler characteristic $-1$ and satisfies the hypotheses of Lemma \ref{standard-form}, so the linear system is empty.
					
					\item $d=4$, $m=1$, $m'=2$.
					
					$|B|=|4H-\sum_{j=1}^{10}E_{j}-2E_{i}|$ has Euler characteristic $0$, and again Lemma \ref{standard-form} applies so the linear system is empty.	
				\end{itemize}
			\end{itemize}
			\begin{remark}
				The proof of Proposition \ref{elliptic-curves} assumes that the $10$ blown-up points are general in the sense that Condition \ref{generality-base-loci} holds for all $|D|$ listed in Lemma \ref{linear-systems}.
			\end{remark}

		\end{proof}
		\begin{theorem}\label{blowdown}
			The data of $i:\mathcal{A}\hookrightarrow D^{b}_{\textnormal{Coh}}(X)$ uniquely characterizes the blowdown map $\pi\colon X\to\mathbf{P}^{2}$.
		\end{theorem}
		\begin{proof}
			By Proposition \ref{elliptic-curves}, the given data characterizes the divisors of the form $-K_{X}+E_{i}$, which identify the $E_{i}$'s, which uniquely determines the blowdown map.
		\end{proof}
		Note that the same scheme $X$ can have multiple blow-down maps to $\mathbf{P}^{2}$, for instance by taking Cremona transformations. Thus, we have actually recovered some non-trivial geometry from the phantom. 
		
		However, this result depended on having the data of the embedding of the phantom, and is still far from using only the intrinsic properties of $\mathcal{A}$ as a triangulated category. We hope the following class of objects may prove promising in that direction.
		
	\section{Objects with Two Non-zero Cohomology Sheaves}\label{section-two-cohomology-sheaves}
	
		In this section we will construct and study objects of the phantom that we can describe very explicitly. It is here that our modification of Krah's phantom construction is important: the analogous objects in his phantom have more than two non-zero cohomology sheaves!  

		The idea is to construct a sheaf $\mathcal{G}$ that is orthogonal to most of the $\mathcal{E}_{i}$'s, so that its projection into the phantom is not much more complicated. We will take $\mathcal{G}$ to be a line bundle $\mathcal{L}$ supported on a carefully chosen curve $\iota:C\hookrightarrow X$.
		
		While the previous two constructions were rather general, the following will engage fundamentally with the particular geometry of $X$.
				
		\subsection{Existence of special curves}	
		
		We want $\text{Hom}^{*}(\mathcal{O}_{X},\iota_{*}\mathcal{L})=\text{Hom}^{*}(\mathcal{O}_{X}(D_{i}),\iota_{*}\mathcal{L})=0$. For this to hold, we need $D_{i}.C=0$, implying $C$ is linearly equivalent to a multiple of $F$. So we want to find some $|-nF|$ that contains curves. Since $h^{0}(X,\mathcal{O}_{X}(-F))=h^{0}(X,\mathcal{O}_{X}(-2F))=0$, we will need $n\geq 3$. 
		
		\begin{proposition}\label{ample}
			$-F$ is ample. 
		\end{proposition}
		\begin{proof}
			To apply the Nakai-Moishezon criterion, we note $(-F)^{2}=1>0$, and we need to show for all reduced irreducible curves $C\subseteq X$ that $-F.C>0$. Since $-F.E_{i}=6>0$, it suffices to consider $C\sim aH-\sum_{i=1}^{10}b_{i}E_{i}$. Since $-F.C=19a-6\sum_{i=1}^{10}b_{i}$, we want to show 
			$$\frac{\sum_{i=1}^{10}b_{i}}{a}<\frac{19}{6} $$
			
			By Lemma \ref{symmetry}, for each $\sigma\in S_{10}$ there exists a curve $C_{\sigma}\sim aH-\sum_{i=1}^{10}b_{\sigma(i)}E_{i}$. Let $\tilde{C}$ be the union of all of them, so 
			$$\tilde{C}\in\left|10!aH-\left(9!\sum_{j=1}^{10}b_{j}\right)\sum_{i=1}^{10}E_{i} \right| $$
			Since this is a non-empty homogeneous system, we have from Lemma \ref{non-special} the bound 
			$$\frac{10!a}{9!\sum_{j=1}^{10}b_{j}}=\frac{10a}{\sum_{j=1}^{10}b_{j}}\geq\frac{2280}{721} $$
			So we have 
			$$\frac{\sum_{j=1}^{10}b_{j}}{a}\leq\frac{7210}{2280}<\frac{19}{6} $$
			as desired.
			
			Since we have applied Lemma \ref{symmetry} to countably many divisor classes, at this point we have only shown $-F$ is ample assuming $\{x_{i}\}_{i=1}^{10}$ are in very general position. However, by Theorem 1.2.17 of \cite{positivity}, ampleness in one configuration implies ampleness in an open subset of configurations, which implies the result for $\{x_{i}\}_{i=1}^{10}$ in general position.
		\end{proof} 
		\begin{remark}
			Since we applied a general openness result, we do not have a precise description of the generality condition on $\{x_{i}\}_{i=1}^{10}$ assumed in the proof of Proposition \ref{ample}.
		\end{remark}
		\begin{corollary}\label{existence}
			With the same assumptions as Proposition \ref{ample} and for large enough $n$, general curves in $|-nF|$ are integral and non-singular, and their pushforwards intersect each $x_{i}$ with multiplicity exactly $6n$.
		\end{corollary}
		\begin{proof}
			The first part follows from Bertini's theorem (Hartshorne Lemma V.1.2), and the latter is generally known (see \cite{DP2023}[Remark 1.2]).
		\end{proof}
		The genus of such a curve is 
		$$g_{n}=\frac{1}{2}(19n-1)(19n-2)-10\frac{(6n)(6n-1)}{2}=\frac{n^{2}+3n+2}{2}$$
		So a general line bundle of degree $\frac{1}{2}n(n+3)$ will do.
		
		Now we show that $|-3F|$ contains a reduced irreducible curve, and this will require $\{x_{i}\}_{i=1}^{10}$ to only be in general, not very general position. 		
		\begin{lemma}\label{unique}
			$\dim|-3F|=0$.
		\end{lemma}
		\begin{proof}
			Note $-3F=57H-18\sum_{i=1}^{10}E_{i}$, so 
			$$\chi(-3F)=\binom{57+2}{2}-10\binom{18+1}{2}=1,$$
			and by Lemma \ref{174/55}, the actual dimension matches the expected dimension.
		\end{proof}
		Let $C$ be the unique (not necessarily integral) curve in $|-3F|$. 
		
		\begin{lemma}\label{57H-18E}
			If $D=dH-m\sum_{i=1}^{10}E_{i}$, with $0<d<57$ and $0\leq m\leq 18$, then either $|D|$ or $|-3F-D|$ is empty.
		\end{lemma}
		\begin{proof}
			Recalling Lemma \ref{non-special}, the only integer choices such that $\frac{d}{m},\frac{57-d}{18-m}\geq\frac{2280}{721}$ are $(d,m)=(19,6),(38,12)$, and we know $H^{0}(X,\mathcal{O}_{X}(-F))=H^{0}(X,\mathcal{O}_{X}(-2F))=0$ from Proposition \ref{Krah}.
		\end{proof}
		\begin{lemma}\label{exactly}
			$\pi_{*}C$ vanishes to order exactly $18$ at each $x_{i}$.
		\end{lemma}
		\begin{proof}
			We are out of the range of the hypotheses of Lemma \ref{standard-form}, so we need a different argument. Assuming the contrary, $|-3F-E_{i}|$ is nonempty for some $i$, and so for all $i$, by Lemma \ref{symmetry}. Note that all $|-3F-E_{i}|\hookrightarrow |-3F|$, and since $\dim|-3F|=0$, these are all isomorphisms. Thus, $C$ vanishes to order $19$ at all $x_{i}$, so $C\in|57H-19\sum_{i=1}^{10}E_{i}|$. But by Lemma \ref{non-special}, that linear system is empty, so we have a contradiction.  
		\end{proof}
		\begin{lemma}\label{|-3F|}
			$C$ is reduced and irreducible.
		\end{lemma}
		\begin{proof}
			Assume the contrary. Let $B$ be a reduced irreducible component of $C$, and let $B\sim dH-\sum_{j=1}^{10}m_{j}E_{j}$, so $|\sigma B|$ and $|-3F-\sigma B|$ are nonempty for all $\sigma\in S_{10}$ by Lemma \ref{symmetry} (by Lemma \ref{permutation} we can abuse notation to let $\sigma B$ be the unique effective divisor in $\sigma|B|$). 
			
			Any pair of curves respectively in $|\sigma B|$ and $|-3F-\sigma B|$ have product in $|-3F|$, but since $C\in|-3F|$ is the unique curve, any $C_{\sigma}\in|\sigma B|$ must be a component of $C$. Let $\widetilde{B}=\bigcup_{\sigma B}\sigma B$.
			
			Since both $|\widetilde{B}|$ and $|-3F-\widetilde{B}|$ are nonempty, by Lemma \ref{57H-18E} we must have $\widetilde{B}\sim-3F$. Then $|\{\sigma B\}|d=57$. But $\{b_{i}\}_{i=1}^{10}$ will never have exactly $3$, $19$, or $57$ permutations, so we must have $d=57$ and $|\{\sigma B\}|=1$. 
			
			Thus, $B=\widetilde{B}=C$, and $C$ is reduced and irreducible. 
		\end{proof}
		\begin{remark}
			The proof of Lemma \ref{unique} assumes the $10$ blown-up points are general in the sense that Condition \ref{generality} holds for $|-3F|$, Lemma \ref{57H-18E} and Lemma \ref{|-3F|} assume Condition \ref{generality-base-loci} holds for $|-3F|$, and Lemma \ref{exactly} assumes Condition \ref{generality} holds for each $|-3F-E_{i}|$.
		\end{remark}
		
		The arithmetic genus of $C$ is $\frac{1}{2}(57-1)(57-2)-10\frac{(18)(17)}{2}-10$, but it is possible that there could be singularities. However, this will not be an issue for our construction.
	
	\subsection{Cohomology sheaves and Homs}
		From now on, let $C$ be a smooth genus $g$ curve, let $\iota:C\to X$ be a map with (possibly singular) image in $|-nF|$ for $n\geq 3$, let $\mathcal{L}\in\textnormal{Pic}^{g-1}(C)$ be generic, let $\mathcal{G}:=\iota_{*}\mathcal{L}$, and let $\mathcal{P}:=i^{*}\mathcal{G}$. For instance, we could take $C$ to be the normalization of the unique curve in $|-3F|$. We will study the cohomology sheaves and Homs of $\mathcal{P}$.
		
		\begin{lemma}\label{concentrated-cohomology}
			$\textnormal{Hom}^{*}(\mathcal{E}_{i},\mathcal{G})$ vanishes for $i>2$ and is concentrated in degree $1$ for $i=1,2$. 
		\end{lemma}
		\begin{proof}			
			Note that 
			\begin{align*}
				\text{Hom}^{*}(\mathcal{O}_{X}(D),\mathcal{G}) &= H^{*}(C,\mathcal{L}(-D.C)) \\&=
				H^{*}(C,\mathcal{L}(nF.D))			
			\end{align*}
			A generic line bundle on $C$ of degree $d<g$ has no global sections, since the image of $C^{d}\to\text{Pic}^{g-1}(C)$ corresponds to classes with an effective divisor, but the domain has dimension $d$ and the codomain has strictly larger dimension $g$.
			
			If $D=0,-D_{i}$, then $nF.D=0$, so $\mathcal{L}(nF.D)$ is a generic line bundle of degree $g-1$ and so by the above and Riemann-Roch, $H^{0}(C,\mathcal{L}(nF.D))=H^{1}(\mathcal{L}(C,nF.D))=0$. If $D=-kF$, then $nF.D=-nk<0$, so $\mathcal{L}(nF.D)$ is a generic line bundle of degree less than $g-1$, and so has cohomology concentrated in degree $1$.  
		\end{proof}
	
		Now we are ready to give an extremely explicit description of $\mathcal{P}$.
		
		\begin{proposition}\label{two-cohomology-sheaves}
			$\mathcal{H}^{i}(\mathcal{P})$ is zero for $i\neq 0,1$, with 
			$$\mathcal{H}^{1}(\mathcal{P})=\mathcal{O}_{X}(-2F)\otimes H^{2}(X,\mathcal{O}_{X}(F))\otimes H^{1}(X,\mathcal{G}(F))$$
			and $\mathcal{H}^{0}(\mathcal{P})$ fits in the exact sequence
			$$0\to\mathcal{G}\to\mathcal{H}^{0}(\mathcal{P})\to\bigoplus_{k=1,2}\mathcal{O}_{X}(-kF)\otimes H^{1}(\mathcal{G}(kF))\to 0$$
		\end{proposition}
		\begin{proof}
			We will use Proposition \ref{projection-complex} to compute the projection $i^{*}\mathcal{G}=\mathcal{P}$. 
			
			Let $\mathcal{E}_{1}^{\bullet},\mathcal{E}_{2}^{\bullet},\mathcal{G}^{\bullet}$ be quasi-coherent injective resolutions for $\mathcal{E}_{1}=\mathcal{O}_{X}(-2F) $, $\mathcal{E}_{2}=\mathcal{O}_{X}(-F)$, and $\mathcal{G}$ respectively.	For $k=1,2$, let $i_{k}$ index a basis for $\text{Hom}^{1}(\mathcal{O}_{X}(-kF),\mathcal{G})=H^{*}(X,\mathcal{G}(kF))$, and let $j$ index a basis for $\text{Hom}^{2}(\mathcal{O}_{X}(-2F),\mathcal{O}_{X}(-F))=H^{2}(X,\mathcal{O}_{X}(F))$. 
			
			Then by Proposition \ref{projection-complex} and Lemma \ref{concentrated-cohomology},  the totalization of the following double complex (which we have rotated 90 degrees clockwise for typographical convenience) is quasi-isomorphic to $\mathcal{P}$:
			$$\begin{tikzcd}
				&                                                                                                    &                                                                                                    & {\bigoplus_{i_{1},j}\mathcal{E}_{1}^{0}} \arrow[d] \arrow[r]                                       & ... \\
				& \bigoplus_{i_{1}}\mathcal{E}_{1}^{0}\oplus\bigoplus_{i_{2}}\mathcal{E}_{2}^{0} \arrow[d] \arrow[r] & \bigoplus_{i_{1}}\mathcal{E}_{1}^{1}\oplus\bigoplus_{i_{2}}\mathcal{E}_{2}^{1} \arrow[d] \arrow[r] & \bigoplus_{i_{1}}\mathcal{E}_{1}^{2}\oplus\bigoplus_{i_{2}}\mathcal{E}_{2}^{2} \arrow[d] \arrow[r] & ... \\
				\mathcal{G}^{0} \arrow[r] & \mathcal{G}^{1} \arrow[r]                                                                                    & \mathcal{G}^{2} \arrow[r]                                                                                    & \mathcal{G}^{3} \arrow[r]                                                                                    & ...
			\end{tikzcd}$$
			where the horizontal maps are induced by the maps of the resolutions and the vertical maps correspond to the basis elements described above.
			
			Call the totalization $T^{\bullet}$. Note 
			$$\mathcal{T}^{m}=\mathcal{G}^{m}\oplus\bigoplus_{i_{1}}\mathcal{E}_{1}^{m}\oplus\bigoplus_{i_{2}}\mathcal{E}_{2}^{m}\oplus\bigoplus_{i_{1},j}\mathcal{E}_{1}^{m-1} $$
			so $\mathcal{H}^{m}(\mathcal{T}^{\bullet})$ is zero for $m\geq 2$ since each of $\mathcal{E}_{1}^{\bullet},\mathcal{E}_{2}^{\bullet},\mathcal{G}^{\bullet}$ is exact in degrees $m\geq 1$, and similarly, 
			$$\mathcal{H}^{1}(\mathcal{T}^{\bullet})=\bigoplus_{i_{1},j}\mathcal{H}^{0}\left(\mathcal{E}_{1}^{\bullet}\right)=\bigoplus_{i_{1},j}\mathcal{E}_{1} $$
			Finally, 
			$$\mathcal{H}^{0}(\mathcal{T}^{\bullet})=\ker\left(\mathcal{G}^{0}\oplus\bigoplus_{i_{1}}\mathcal{E}_{1}^{0}\oplus\bigoplus_{i_{2}}\mathcal{E}_{2}^{0}\to\mathcal{G}^{1}\oplus\bigoplus_{i_{1}}\mathcal{E}_{1}^{1}\oplus\bigoplus_{i_{2}}\mathcal{E}_{2}^{1}\right) $$
			Note that $\mathcal{G}\to\mathcal{G}^{0}$ induces a map into this kernel, and taking the quotient by its image we have 
			\begin{align*}
				\text{coker}(\mathcal{G}\to\mathcal{H}^{0}(\mathcal{T}^{\bullet})) &=
				\ker\left(\bigoplus_{i_{1}}\mathcal{E}_{1}^{0}\oplus\bigoplus_{i_{2}}\mathcal{E}_{2}^{0}\to\bigoplus_{i_{1}}\mathcal{E}_{1}^{1}\oplus\bigoplus_{i_{2}}\mathcal{E}_{2}^{1}\right) \\&=
				\bigoplus_{i_{1}}\mathcal{E}_{1}\oplus\bigoplus_{i_{1}}\mathcal{E}_{2}
			\end{align*}			 
			where the last equality holds by exactness of the resolutions $\mathcal{E}_{k}\to\mathcal{E}_{k}^{\bullet}$.
			
		\end{proof}
		\begin{remark}
			Since $\mathcal{G}\in\langle\mathcal{E}_{3},...,\mathcal{E}_{10}\rangle^{\perp}$, its projection into the phantom could be computed using the two evaluation triangles to project $\mathcal{G}$ past $\mathcal{E}_{0}$ and $\mathcal{E}_{1}$. Though such a proof would be somewhat more involved than appealing to Proposition \ref{projection-complex}, it would be a lower-machinery proof that the exceptional collection is not full and that the phantom exists.
		\end{remark}
%		\begin{lemma}\label{swap}
%			$j^{!}=S_{X}^{-1}\circ S_{\mathcal{B}^{\perp}}\circ i^{*}$.
%		\end{lemma}
%		\begin{proof}
%			Noting that $S_{X}|_{{}^{\perp}\mathcal{B}}$ maps ${}^\mathcal{B}$ isomorphically to $\mathcal{B}^{\perp}$, we see the right hand side maps into ${}^{\perp}\mathcal{B}$. 			Now we will show that it is the right adjoint of $j$. Let $M\in{}^{\perp}\mathcal{B}$ and $G\in D^{b}_{\text{Coh}}(X)$.		
%			\begin{align*}
%				\text{Hom}_{{}^{\perp}\mathcal{B}}(M,S_{X}^{-1} S_{\mathcal{B}^{\perp}}i^{*}G) &=
%				\text{Hom}_{\mathcal{B}^{\perp}}(S_{X}M,S_{\mathcal{B}^{\perp}}i^{*}G) \\&=
%				\text{Hom}_{\mathcal{B}^{\perp}}(i^{*}G,S_{X}M)^{\vee} \\&=
%				\text{Hom}_{X}(G,S_{X}M)^{\perp} \\&=
%				\text{Hom}_{X}(M,G)
%			\end{align*}
%		\end{proof}
		The following will show that $\mathcal{P}$ is a simple, universally gluable object.
		\begin{proposition}\label{Hom(P,P)}
			$$\textnormal{Hom}_{X}^{*}(\mathcal{P},\mathcal{P})=\begin{cases}
				\mathbf{C} & *=0 \\
				\textnormal{Ext}_{X}^{1}(\iota_{*}\mathcal{L},\iota_{*}\mathcal{L}) & *=1 \\
				\mathbf{C}^{4n^{2}}\oplus\textnormal{Ext}_{X}^{2}(\iota_{*}\mathcal{L},\iota_{*}\mathcal{L}) & *=2 \\
				\mathbf{C}^{3n^{2}} & *=3 \\
				0 & \textnormal{else} \\
			\end{cases} $$
		\end{proposition}
		\begin{proof}
			By Lemma \ref{concentrated-cohomology},  $\text{Hom}_{X}^{*}(\mathcal{E}_{i},\iota_{*}\mathcal{L})=0$ for $\mathcal{E}_{i}\neq \mathcal{O}_{X}(-F),\mathcal{O}_{X}(-2F)$. Furthermore, for $k=1,2$ we have 
			\begin{align*}
				\text{Hom}_{X}^{*}(\mathcal{O}_{X}(-kF),\iota_{*}\mathcal{L}) &=
				H^{*}(X,\iota_{*}\mathcal{L}(kF)) \\&=
				H^{*}(C,\mathcal{L}(kF.C))\\&=
				H^{*}(C,\mathcal{L}(-nk)) \\&=
				\mathbf{C}^{nk}[-1]
			\end{align*} 
			\begin{align*}
				\text{Hom}_{X}^{*}(\iota_{*}\mathcal{L},\mathcal{O}_{X}(-kF)) &=
				\text{Hom}_{X}^{2-*}(\mathcal{O}_{X}(-kF),\iota_{*}\mathcal{L}\otimes\omega_{X}) \\&= 
				H^{2-*}(X,\iota_{*}\mathcal{L}(kF+K_{X})) \\&=
				H^{2-*}(C,\mathcal{L}((kF+K_{X}).C))\\&=
				H^{2-*}(C,\mathcal{L}(n(3-k))) \\&=
				\mathbf{C}^{n(3-k)}[-2]
			\end{align*}		
			Thus, applying the spectral sequence of Corollary \ref{spectral-sequence} to $\text{Hom}^{*}(i^{*}\mathcal{G},i^{*}\mathcal{G})$, the $E_{1}^{0,q}$ terms correspond to $\text{Hom}_{X}^{*}(\iota_{*}\mathcal{L},\iota_{*}\mathcal{L})$, and the only other nonzero terms are: 
			$$E_{1}^{-1,3}=\bigoplus_{k=1,2}\text{Hom}_{X}^{1}(\mathcal{O}_{X}(-kF),\iota_{*}\mathcal{L})\otimes\text{Hom}_{X}^{2}(\iota_{*}\mathcal{L},\mathcal{O}_{X}(-kF))=\mathbf{C}^{4n^{2}}$$
			$$E_{1}^{-2,5}=\text{Hom}_{X}^{1}(\mathcal{O}_{X}(-F),\iota_{*}\mathcal{L})\otimes\text{Hom}^{2}(\mathcal{O}_{X}(-2F),\mathcal{O}_{X}(-F))\otimes\text{Hom}_{X}^{2}(\iota_{*}\mathcal{L},\mathcal{O}_{X}(-2F))=\mathbf{C}^{3n^{2}} $$ 
			so all differentials vanish for degree reasons. Finally, note 
			\begin{align*}
				\text{Hom}_{X}(\iota_{*}\mathcal{L},\iota_{*}\mathcal{L}) &=
				\text{Hom}_{C}(\iota^{*}\iota_{*}\mathcal{L},\mathcal{L}) \\&=
				\text{Hom}_{C}(\mathcal{L},\mathcal{L}) \\&=
				\mathbf{C}
			\end{align*}
		\end{proof}
		\begin{corollary}
			Define $\mathcal{P}'$ by the same construction as $\mathcal{P}$ but with a different choice of $\iota',\mathcal{C}',\mathcal{L}'$. Then 
			$$\textnormal{Hom}_{X}^{*}(\mathcal{P},\mathcal{P}')=\begin{cases}
				0 & *=0 \\
				\textnormal{Ext}_{X}^{1}(\iota_{*}\mathcal{L},\iota'_{*}\mathcal{L}') & *=1 \\
				\mathbf{C}^{4n^{2}}\oplus\textnormal{Ext}_{X}^{2}(\iota_{*}\mathcal{L},\iota'_{*}\mathcal{L}') & *=2 \\
				\mathbf{C}^{3n^{2}} & *=3 \\
				0 & \textnormal{else} \\
			\end{cases} $$
		\end{corollary}
		\begin{proof}
			We have $\text{Hom}_{X}^{*}(\mathcal{E}_{i},\iota'_{*}\mathcal{L}')=\text{Hom}_{X}^{*}(\mathcal{E}_{i},\iota_{*}\mathcal{L})$ and $\text{Hom}_{X}^{*}(\iota'_{*}\mathcal{L}',\mathcal{E}_{i})=\text{Hom}_{X}^{*}(\iota_{*}\mathcal{L},\mathcal{E}_{i})$, so the spectral sequence computation is the same as in the proof of Proposition \ref{Hom(P,P)}. The only difference is 
			\begin{align*}
				\text{Hom}_{X}(\iota_{*}\mathcal{L},\iota'_{*}\mathcal{L}') &=
				\text{Hom}_{C'}(\iota'^{*}\iota_{*}\mathcal{L},\mathcal{L}') \\&=
				\begin{cases}
					\text{Hom}_{C'}((\iota_{*}\mathcal{L})|_{C\cap C'},\mathcal{L}') & C\neq C' \\
					\text{Hom}_{C'}(\mathcal{L},\mathcal{L}') & C=C'
				\end{cases} \\&=
				0
			\end{align*}
			since there are no maps from a torsion sheaf to a locally free sheaf, and no maps between distinct line bundles of the same degree on a curve.
		\end{proof}

	\subsection{Deformation theory}
		Again we have that $\mathcal{P}$ is a simple universally-gluable object and so lies in the algebraic stack $s\mathcal{M}(\mathcal{A})$ of such objects.
	
		Furthermore, $\text{Hom}^{1}(\mathcal{P},\mathcal{P})$ corresponds to deformations of $\mathcal{P}$, so equality with $\text{Ext}^{1}(\mathcal{G},\mathcal{G})$ means that our objects form an open substack of $s\mathcal{M}(\mathcal{A})$.
		
		We will show this precisely.
		
		\begin{lemma}
			For $\iota:C\to X$ a closed immersion and $\mathcal{L},\mathcal{M}\in\textnormal{Pic}(C)$, there is an isomorphism
			$$(L^{j}\iota^{*}\iota_{*}\mathcal{L})\otimes\mathcal{M}\xrightarrow{\sim}L^{j}\iota^{*}\iota_{*}(\mathcal{L}\otimes\mathcal{M})$$
		\end{lemma}
		\begin{proof}
			The map is induced by adjunction via
			$$\mathcal{M}\to\mathcal{H}\text{om}(\mathcal{L},\mathcal{L}\otimes\mathcal{M})\to\mathcal{H}\text{om}(L^{j}\iota^{*}\iota_{*}\mathcal{L},L^{j}\iota^{*}\iota_{*}(\mathcal{L}\otimes\mathcal{M})) $$
			Since taking stalks commutes with pullbacks and the non-empty fibres of $i$ are points, the map on the stalk for $c\in C$ is $\mathcal{O}_{c}\to\mathcal{O}_{c}$, and since it is nonzero it is an isomorphism.  
		\end{proof}
		\begin{proposition}\label{normal-bundle}
			There is an exact sequence 
			$$0\to H^{1}(C,\mathcal{O}_{C})\to\textnormal{Ext}_{X}^{1}(\iota_{*}\mathcal{L},\iota_{*}\mathcal{L})\to H^{0}(C,\mathcal{N}_{X/C})\to 0$$
			where $\mathcal{N}_{X/C}$ is the normal bundle. 
		\end{proposition}
		\begin{proof}
			Note that $\iota_{*}$ is exact since $\iota$ is a closed immersion. Consider the Grothendieck spectral sequence composing $\text{Ext}_{C}^{\bullet}(-,\mathcal{L})$ and $L\iota^{*}\iota_{*}$ applied to $\mathcal{L}$. Noting that $\text{Ext}_{C}^{2}=0$, we have the following exact sequence of low degrees:
			$$0\to\text{Ext}_{C}^{1}(\iota^{*}\iota_{*}\mathcal{L},\mathcal{L})\to 
			\text{Ext}_{C}^{1}(L\iota^{*}i_{*}\mathcal{L},\mathcal{L})\to
			\text{Ext}_{C}^{0}(L^{1}\iota^{*}\iota_{*}\mathcal{L},\mathcal{L})\to 0$$
			
			By the lemma, the first term is isomorphic to $\text{Ext}_{C}^{1}(\iota^{*}\iota_{*}\mathcal{O}_{C},\mathcal{O}_{C})$, and since $\iota^{*}\iota_{*}\mathcal{O}_{C}\cong\mathcal{O}_{C}$, it is isomorphic to $H^{1}(C,\mathcal{O}_{C})$. Similarly, the last term is isomorphic to $\text{Hom}_{C}^{0}(L^{1}\iota^{*}\iota_{*}\mathcal{O}_{C},\mathcal{O}_{C})$, and we have a flat resolution 
			$$0\to\mathcal{O}_{X}(-C)\to\mathcal{O}_{X}\to \iota_{*}\mathcal{O}_{C}\to 0$$
			So $L\iota^{*}\iota_{*}\mathcal{O}_{C}$ is quasi-isomorphic to $\iota^{*}\mathcal{O}_{X}(-C)\to \iota^{*}\mathcal{O}_{X}$. But this is the zero map, so $L^{1}\iota^{*}\iota_{*}\mathcal{O}_{C}\cong\mathcal{O}_{X}(-C)|_{C}\cong N_{X/C}^{\vee}$. And so the last term of the short exact sequence is $\text{Hom}_{C}(N_{X/C}^{\vee},\mathcal{O}_{C})\cong H^{0}(C,N_{X/C})$.
		\end{proof}	
		Thus, deformations of $\mathcal{G}=\iota_{*}\mathcal{L}$ arise from deformations of $\mathcal{L}$ and of $C$.
		
		\begin{remark}
			By Lemma \ref{unique}, there is a unique element of $C\in|-3F|$, so the moduli of all corresponding $\mathcal{P}$ comprises deformations of $\mathcal{L}$, which form an open subset of $\text{Pic}^{g-1}(C)$. Thus, if we could intrinsically characterize this $\mathcal{P}\in\mathcal{A}$ using only the triangulated category structure, we could then identify $\text{Pic}^{g-1}(C)$ up to birational equivalence, which is enough to identify $C$. This $C$ is geometrically rich, and could perhaps even uniquely identify $X$ among all blowups of $\mathbf{P}^{2}$ in 10 points. Indeed, $C$ has expected genus $10$, and the moduli of genus $10$ curves has dimension $27$, which is greater than the $20$ dimensions of moduli of $10$ points in $\mathbf{P}^{2}$.
			
			The problem of intrinsically identifying objects of $\mathcal{A}$ seems quite difficult. But observing that $\text{Hom}^{*}(\mathcal{P},\mathcal{P})$ is quite small, perhaps it could be shown that $\mathcal{P}$ minimizes $\text{Hom}^{*}(-,-)$ among all objects of the phantom in some precise way, which could lead to an intrinsic characterization. 
		\end{remark}

	\nocite{*}
	\footnotesize{\bibliographystyle{plain}\bibliography{phantom_objects}}

\end{document}